\documentclass[12pt]{article}

\usepackage{amsmath}
\usepackage{amssymb}
\usepackage{amsthm}
\usepackage{bbm}
\usepackage{hyperref}
\usepackage{enumerate}

\newtheorem{thm}{Theorem}[section]

\newtheorem{lem}[thm]{Lemma}

\newtheorem{assumption}[thm]{Assumption}

\newtheorem{definition}[thm]{Definition}

\newtheorem{example}[thm]{Example}

\newtheorem{remark}[thm]{Remark}
\newenvironment{rem}{\begin{remark}\rm}{\end{remark}}

\newcommand{\EX}{{\mathbb{E}}}
\newcommand{\PX}{{\mathbb{P}}}

\title{Information-Theoretic Extensions of the Shannon-Nyquist Sampling Theorem}

\author{\small \begin{tabular}{ccc}
Xianming Liu & Guangyue Han\\
Huazhong University of Science and Technology & The University of Hong Kong\\
email: xmliu@hust.edu.cn & email: ghan@hku.hk\\
\end{tabular}}

\date{{\normalsize \today}}

\begin{document} \maketitle

\begin{abstract}
A continuous-time white Gaussian channel can be formulated using a white Gaussian noise, and a conventional way for examining such a channel is the sampling approach based on the classical Shannon-Nyquist sampling theorem, where the original continuous-time channel is converted to an equivalent discrete-time channel, to which a great variety of established tools and methodology can be applied. However, one of the key issues of this scheme is that continuous-time feedback cannot be incorporated into the channel model. It turns out that this issue can be circumvented by considering the Brownian motion formulation of a continuous-time white Gaussian channel. Nevertheless, as opposed to the white Gaussian noise formulation, a link that establishes the information-theoretic connection between a continuous-time white Gaussian channel under the Brownian motion formulation and its discrete-time counterparts has long been missing. This paper is to fill this gap by establishing information-theoretic extensions of the Shannon-Nyquist theorem, which naturally yield causality-preserving connections between continuous-time Gaussian feedback channels and their associated discrete-time versions in the forms of sampling and approximation theorems. As an example of the possible applications of the extensions, we use the above-mentioned connections to analyze the capacity of a continuous-time white Gaussian feedback channel.
\end{abstract}

\section{Introduction}  \label{intro}

Continuous-time Gaussian channels were considered at the very inception of information theory. In his celebrated paper~\cite{sh48} birthing information theory, Shannon studied the following continuous-time white Gaussian channels:
\begin{equation} \label{white-Gaussian-noise-channel}
Y(t)=X(t)+Z(t), \quad t \in \mathbb{R},
\end{equation}
where the stochastic process $\{X(t)\}$ is the {\em channel input} with average power limit $P$, $\{Z(t)\}$ is the white Gaussian noise with flat power spectral density $1$ and $\{Y(t)\}$ is the {\em channel output}. Shannon actually only considered the case that the channel has bandwidth limit $\omega$, namely, the support of the fourier transform of $\{X(t)\}$ is contained in $[-w, w]$. Using the celebrated Shannon-Nyquist sampling theorem~\cite{ny24, sh49}, the continuous-time channel (\ref{white-Gaussian-noise-channel}) can be equivalently represented by a parallel Gaussian channel:
\begin{equation} \label{sampled-white-Gaussian-noise-channel}
Y_{n}^{(\omega)}=X_{n}^{(\omega)}+Z_{n}^{(\omega)}, \quad n \in \mathbb{Z},
\end{equation}
where the noise process $\{Z_n^{(\omega)}\}$ is i.i.d. with variance $1$~\cite{co2006}. Regarding the ``space'' index $n$ as time, the above parallel channel can be interpreted as a discrete-time Gaussian channel associated with the continuous-time channel (\ref{white-Gaussian-noise-channel}). It is well known from the theory of discrete-time Gaussian channels that the capacity of the channel (\ref{sampled-white-Gaussian-noise-channel}) can be computed as
\begin{equation} \label{finite-bandwidth-capacity}
C^{(\omega)}=\omega \log \left(1+\frac{P}{2 \omega}\right).	
\end{equation}
Then, the capacity $C$ of the channel (\ref{white-Gaussian-noise-channel}) can be computed by taking the limit of the above expression as $\omega$ tends to infinity:
\begin{equation} \label{infinite-bandwidth-capacity}
C=\lim_{\omega \to \infty} C^{(\omega)}=P/2.
\end{equation}

The {\em sampling approach} consisting of (\ref{white-Gaussian-noise-channel})-(\ref{infinite-bandwidth-capacity}) as above, which serves as a link between the continuous-time channel (\ref{white-Gaussian-noise-channel}) and the discrete-time channel (\ref{sampled-white-Gaussian-noise-channel}), typifies a conventional way to examine continuous-time Gaussian channels: convert them into associated discrete-time Gaussian channels, for which we have ample ammunition at hands. Moments of reflection, however, reveals that the sampling approach for the channel capacity (with bandwidth limit or not) is heuristic in nature: For one thing, a bandwidth-limited signal cannot be time-limited, which renders it infeasible to define the data transmission rate if assuming a channel has bandwidth limit. In this regard, rigorous treatments coping with this issue and other technicalities can be found in~\cite{wy66, ga68}; see also~\cite{sl76} for a relevant in-depth discussion. Another issue is that, even disregarding the above technical nuisance arising from the bandwidth limit assumption, the sampling approach only gives a lower bound on the capacity of (\ref{white-Gaussian-noise-channel}): it shows that $P/2$ is achievable via a class of special coding schemes, but it is not clear that why transmission rate higher than $P/2$ cannot be achieved by other coding schemes. The capacity of (\ref{white-Gaussian-noise-channel}) was rigorously studied in~\cite{fo61, be62}, and a complete proof establishing $P/2$ as its capacity can be found in~\cite{as63,as64}.

Alternatively, the continuous-time white Gaussian channel (\ref{white-Gaussian-noise-channel}) can be examined~\cite{ih93} under the Brownian motion formulation:
\begin{equation} \label{Brownian-motion-channel}
Y(t)=\int_0^t X(s)ds + B(t), \quad t \in \mathbb{R}_+
\end{equation}
where slightly abusing the notation, we still use $\{Y(t)\}$ to denote the output corresponding to the input $\{X(s)\}$, and $\{B(t)\}$ denotes the standard Brownian motion. Here we remark that, the formulation in (\ref{Brownian-motion-channel}) is often regarded as the integral version of that in (\ref{white-Gaussian-noise-channel}) due to the long-held heuristic interpretation of a white Gaussian noise as the ``derivative'' of a Brownian motion, and via a routine orthonormal decomposition argument, both of the two channels are equivalent to a parallel channel consisting of infinitely many Gaussian sub-channels~\cite{as65}.

An immediate and convenient consequence of such a formulation is that many notions in discrete time, including mutual information and typical sets, carry over to the continuous-time setting, which will rid us of the nuisances arising from the bandwidth limit assumption. Indeed, such a framework yields a fundamental formula for the mutual information of the channel (\ref{Brownian-motion-channel})~\cite{Duncan70, ka71} and a clean and direct proof~\cite{ka71} that the capacity of (\ref{Brownian-motion-channel}) is $P/2$; moreover, as evidenced by numerous results collected in~\cite{ih93} on continuous-time Gaussian channels, the use of Brownian motions elevates the level of rigor of our treatment, and equip us with a wide range of established techniques and tools from stochastic calculus. Here we remark that Girsanov's theorem, one of the fundamental theorems in stochastic calculus, lays the foundation of our treatment. We refer to~\cite{li01,ih93}, where Girsanov's theorem (and its numerous variants) and its wide range of applications in information theory are discussed in great details.

Furthermore, the Brownian motion formulation is also versatile enough to accommodate feedback. Here we note that a continuous-time Gaussian channel as in (\ref{Brownian-motion-channel}) can be alternatively written as
\begin{equation} \label{g-formulation}
Y(t)=\int_0^t g(s, M) ds + B(t), \quad t \in \mathbb{R}_+,
\end{equation}
where $M$ is a random variable taking values in a finite alphabet $\mathcal{M}$, interpreted as the message to be transmitted through the channel, and $g$ is a real-valued deterministic function depending on $s \in \mathbb{R}_+$, $M \in \mathcal{M}$, interpreted as the channel input. The formulation in (\ref{g-formulation}) can be readily extended to model a continuous-time white Gaussian channel with feedback, or simply, Gaussian feedback channel, which is be characterized by the following stochastic differential equation (SDE)~\cite{ih93}:
\begin{equation} \label{feedback-memory}
Y(t)=\int_0^t g(s, M, Y_0^{s}) ds + B(t), \quad t \in \mathbb{R}_+,
\end{equation}
where the channel input $g$ also depends on $Y_0^s \triangleq \{Y(r): 0 \leq r \leq s\}$, the channel output up to time $s$ that is fed back to the sender, which will be referred to as the {\em channel feedback}. For obvious reasons, with a same set of constraints, the capacity of the channel (\ref{feedback-memory}) is greater than or equal to that of (\ref{g-formulation}), that is to say, feedback increases the capacity in general. On the other hand though, for much subtler reasons, with the average power limit $P$, the capacity of (\ref{feedback-memory}) is still $P/2$, namely, feedback does not help with the average power constraint~\cite{ka71}.

As opposed to the white Gaussian noise formulation, a feedback channel under the Brownian motion formulation can be naturally translated to the discrete-time setting: the pathwise continuity of a Brownian motion allows the inheritance of temporal causality when the channel is sampled (see Section~\ref{sampling-theorems}) or approximated (see Section~\ref{approximation-theorems}). On the other hand, the white Gaussian noise formulation is facing inherent difficulty as far as inheriting temporal causality is concerned: in converting (\ref{white-Gaussian-noise-channel}) to (\ref{sampled-white-Gaussian-noise-channel}), while $X_n^{(w)}$ are obtained as ``time'' samples of $X(t)$, $Z_n^{(w)}$ are in fact ``space'' samples of $Z(t)$, as they are merely the coefficients of the (extended) Karhunen-Loeve decomposition of $Z(t)$~\cite{ge59, huang, johnson}.

On the other hand though, as opposed to the white Gaussian noise formulation, a link that establishes the information-theoretic connection between the continuous-time channel (\ref{feedback-memory}) and its discrete-time counterparts has long been missing, which may explain why discrete-time and continuous-time information theory (under the Brownian motion formulation) have largely gone separate ways with little interaction for the past several decades. In this paper, we will fill this gap by establishing
information-theoretic extensions of the Shannon-Nyquist theorem, which naturally give causality-preserving connections between continuous-time Gaussian feedback channels and their associated discrete-time versions. We believe the extensions will serve as the above-mentioned missing links and play important roles in the long run for further developing continuous-time information theory, particularly for the communication scenarios when feedback is present.

The remainder of the paper is organized as follows. In Section~\ref{notations}, we introduce our notations and recall some basic notions and facts that will be used in our proofs. In Section~\ref{sampling-theorems}, we prove Theorems~\ref{sampling-theorem}, a sampling theorem for a continuous-time Gaussian feedback channel, which naturally connect such a channel with their sampled discrete-time versions. And in Section~\ref{approximation-theorems}, we prove Theorem~\ref{approximation-theorem}, the so-called approximation theorem, which connects a continuous-time Gaussian feedback channel with its approximated discrete-time versions (in the sense of the Euler-Maruyama approximation~\cite{kl92}). Roughly speaking, a sampling theorem says that a time-sampled channel is ``close'' to the original channel if the sampling is fine enough, and an approximation theorem says that an approximated channel is ``close'' to the original channel if the approximation is fine enough, both in an information-theoretic sense. Note that, as elaborated in Remark~\ref{sampling-approximation-theorems}, the approximation theorem boils down to the sampling theorem when there is no feedback in the channel. In Section~\ref{multi-user-Gaussian-channels}, as an example of the possible applications of the extensions, we use Theorem~\ref{approximation-theorem} to give alternative derivation of the capacity of the channel (\ref{feedback-memory}).

\section{Notations and Preliminaries} \label{notations}

We use $(\Omega,\mathcal{F},\PX)$ to denote the underlying probability space, and $\EX$ to denote the expectation with respect to the probability measure $\PX$. As is typical in the theory of SDEs, we assume the probability space is equipped with a filtration $\{\mathcal{F}_t: 0 \leq t < \infty\}$, which satisfies the {\em usual conditions}~\cite{ka91} and is rich enough to accommodate the standard Brownian motion $\{B(t): 0 \leq t < \infty\}$. Throughout the paper, we will use uppercase letters (e.g., $X$, $Y$, $Y^{(n)}$) to denote random variables, and their lowercase counterparts (e.g., $x$, $y$, $y^{(n)}$) to denote their realizations.

Let $C[0, \infty)$ denote the space of all continuous functions over $[0, \infty)$, and for any $t > 0$, let $C[0, t]$ denote the space of all continuous functions over $[0, t]$. As usual, we will equip the space $C[0, \infty)$ with the filtration $\{\mathcal{B}_t\}_{0 \leq t < \infty}$, where $\mathcal{B}_{\infty}$ denotes the standard Borel $\sigma$-algebra on the space $C[0, \infty)$ and $\mathcal{B}_t = \pi_t^{-1}(\mathcal{B}_{\infty})$, where $\pi_t : C[0, \infty) \to C[0, t]$ is given by the map $(\pi_tx)(s) = x(t\wedge s)$.

For any $\varphi \in C[0, \infty)$, we use $\varphi(\{t_1, t_2, \dots, t_m\})$ to denote $\{\varphi(t_1), \varphi(t_2), \dots, \varphi(t_n)\}$ and $\varphi_0^t$ to denote $\{\varphi(s): 0 \leq s \leq t\}$. The sup-norm of $\varphi_0^t$, denoted by $\|\varphi_0^t\|$, is defined as $\|\varphi_0^t\| = \sup_{0 \leq s \leq t} |\varphi(s)|$; and similarly, we define $\|\varphi_0^t-\psi_0^t\| \triangleq \sup_{0 \leq s \leq t} |\varphi(s)-\psi(s)|$. For any $\varphi, \psi \in C[0, \infty)$, slightly abusing the notation, we define $\|\varphi_0^s-\psi_0^t\| \triangleq \|\hat{\varphi}_0^{\infty}-\hat{\psi}_0^{\infty}\|$, where $\hat{\varphi}, \hat{\psi} \in C[0, \infty)$ are ``stopped'' versions of $\varphi, \psi$ at time $s, t$, respectively, with $\hat{\varphi}(r)=\varphi(r \wedge s)$ and $\hat{\psi}(r)=\psi(r \wedge t)$.

For any two probability measures $\mu$ and $\nu$, we write $\mu\sim\nu$ to mean they are equivalent, namely, $\mu$ is absolutely continuous with respect to $\nu$ and vice versa. For any two processes $X_0^t = \{X(s); 0 \leq s \leq t\}$ and $Y_0^t =\{Y(s); 0 \leq s \leq t\}$, we use $\mu_{X_0^t}$ and $\mu_{Y_0^t}$ to denote the probability distributions on $\mathcal{B}_t$ induced by $X_0^t$ and $Y_0^t$, respectively; and if $\mu_{Y_0^t}$ is absolutely continuous with respect to $\mu_{X_0^t}$, we write the Radon-Nikodym derivative of $\mu_{Y_0^t}$ with respect to $\mu_{X_0^t}$ as $d\mu_{Y_0^t}/d\mu_{X_0^t}$. We use $\mu_{Y_0^t|Z=z}$ denote the probability distribution on $\mathcal{B}_t$ induced by $Y_0^t$ given $Z=z$, and $d\mu_{Y_0^t|Z=z}/d\mu_{X_0^t|Z=z}$ to denote the Radon-Nikodym derivative of $Y_0^t$ with respect to $X_0^t$ given $Z=z$. Obviously, when $Z$ is independent of $X$, $d\mu_{Y_0^t|Z=z}/d\mu_{X_0^t|Z=z}= d\mu_{Y_0^t|Z=z}/d\mu_{X_0^t}$.

We next present some basic notions and facts from information theory and introduce the corresponding notations. For more comprehensive expositions, we refer to~\cite{co2006, ih93}.

Let $X, Y, Z$ be random variables defined on the probability space $(\Omega,\mathcal{F},\PX)$, which will be used to illustrate most of the notions and facts in this section (note that the same notations may have different connotations in other sections). Particularly in this paper, random variables can be discrete-valued with a probability mass function, real-valued with a probability density function or path-valued (more precisely, $C[0, \infty)$-valued or $C[0, t]$-valued).

By definition, for $\EX[X|\sigma(Y, Z)]$, the conditional expectation of $X$ with respect to the $\sigma$-algebra generated by $Y$ and $Z$, there exists a $\sigma(Y) \otimes \sigma(Z)$-measurable function $\Psi(\cdot, \cdot)$ such that $\Psi(Y, Z)=\EX[X|\sigma(Y, Z)]$.
For notational convenience, we will in this paper simply write $\EX[X|\sigma(Y, Z)]$ as $\EX[X|Y, Z]$, and $\Psi(y, z)$ as $\EX[X|y, z]$ and furthermore, $\Psi(Y, z)$ as $\EX[X|Y, z]$.

A {\em partition} of the probability space $(\Omega,\mathcal{F},\PX)$ is a disjoint collection of elements of $\mathcal{F}$ whose union is $\Omega$. It is well known there is a one-to-one correspondence between finite partitions and finite sub-$\sigma$-algebras of $\mathcal{F}$. For a finite sub-$\sigma$-algebra $\mathcal{H} \subset \mathcal{F}$, let $\eta(\mathcal{H})$ denote the corresponding finite partition. The entropy of a finite partition $\xi=\{A_1, A_2, \cdots, A_m\}$, denoted by $H(\xi)$, is defined as $H(\xi)=\sum_{i=1}^m -\PX(A_i)\log \PX(A_i)$, whereas the conditional entropy of $\xi$ given another finite partition $\zeta=\{B_1, B_2, \dots, B_n\}$, denoted by $H(\xi|\zeta)$, is defined as $H(\xi|\zeta)= \sum_{j=1}^n \sum_{i=1}^m -\PX(A_i \cap B_j) \log \PX(A_i|B_j)$. The mutual information between the above-mentioned two partitions $\xi$ and $\zeta$, denoted by $I(\xi; \zeta)$, is defined as $I(\xi; \zeta) = \sum_{j=1}^n \sum_{i=1}^m -\PX(A_i \cap B_j) \log \PX(A_i \cap B_j)/\PX(A_i) \PX(B_j)$.

For the random variable $X$, we define
$$
\eta(X) \triangleq \{\eta(\mathcal{H}): \mathcal{H} \mbox{ is a finite sub-}\sigma\mbox{-algebra of } \sigma(X)\}.
$$
The {\em entropy} of the random variable $X$, denoted by $H(X)$, is defined as
$$
H(X) \triangleq \sup_{\xi \in \eta(X)} H(\xi).
$$
The {\em conditional entropy} of $Y$ given $X$, denoted by $H(Y|X)$, is defined as
$$
H(Y|X) = \inf_{\xi \in \eta(X)} \sup_{\zeta \in \eta(Y)} H(\zeta|\xi).
$$
Here, we note that if $X$ and $Y$ are independent, then obviously it holds that 
\begin{equation} \label{Y-X-Y}
H(Y|X)=H(Y).
\end{equation}
The {\em mutual information} between $X$ and $Y$, denoted by $I(X; Y)$, is defined as
$$
I(X; Y) = \sup_{\xi \in \eta(X), \; \zeta \in \eta(Y)} I(\xi; \zeta).
$$
A couple of properties of mutual information are in order. First, it can be shown, via a concavity argument, that the mutual information is always non-negative. Second, the mutual information is determined by the $\sigma$-algebras generated by the corresponding random variables; more specifically, for any random variables $X', Y', X'', Y''$,
\begin{equation} \label{I=I}
I(X'; Y')=I(X'; Y'') \mbox{ if } \sigma(X') = \sigma(X'') \mbox{ and } \sigma(Y')=\sigma(Y'')
\end{equation}
and
\begin{equation} \label{I<=I}
I(X'; Y') \leq I(X'; Y'') \mbox{ if } \sigma(X') \subset \sigma(X'') \mbox{ and }  \sigma(Y') \subset \sigma(Y'').
\end{equation}
For a quick example, we have $I(X; Y)=I(X, X; Y, Y+X)$ and $I(X; Y) \leq I(X; Y, Z)$.

It turns out that for the case that $X, Y, Z$ are all discrete random variables, all the above-mentioned notions are well-defined and can be computed rather explicitly: $H(X)$ can be computed as $H(X) = \EX[-\log p_X(X)]$, where $p_X(\cdot)$ denotes the probability mass function of $X$; $H(Y|X)$ can be computed as $H(Y|X) = \EX[-\log p_{Y|X}(Y|X)]$, where $p_{Y|X}(\cdot|\cdot)$ denotes the conditional probability mass function of $Y$ given $X$; $I(X; Y)$ can be computed as
\begin{equation} \label{mutual-information}
I(X; Y) = \EX\left[\log \frac{p_{Y|X}(Y|X)}{p_Y(Y)}\right].
\end{equation}
The mutual information is intimately related to entropy. As an example, one verifies that
\begin{equation} \label{I-H}
I(X; Y)=H(Y)-H(Y|X).
\end{equation}
Note that the quality (\ref{I-H}) may fail if non-discrete random variables are involved, since the corresponding entropies $H(Y)$ and $H(Y|X)$ can be infinity. For the case of real-valued random variables with density, this issue can be circumvented using the notion of differential entropy, as elaborated below.

Now, let $Y$ be a real-valued random variable with probability density function $f_Y(\cdot)$. The {\em differential entropy} of $Y$, denoted by $h(Y)$, is defined as $h(Y)= \EX[-\log f_Y(Y)]$. And the {\em differential conditional entropy} of $Y$ given a finite partition $\xi$, denoted by $h(Y|\zeta)$, is defined as $h(Y|\zeta) = \sum_{j=1}^n \PX(A_i) \int f_{Y|A_i}(x) \log f_{Y|A_i}(x) dx$. The {\em differential conditional entropy} of $Y$ given $X$ (which may not be real-valued), denoted by $h(Y|X)$, is defined as $h(Y|X)=\inf_{\xi \in \eta(X)} h(Y|\xi)$; in particular, if the conditional probability density function $f_{Y|X}(\cdot|\cdot)$ exists, then $h(Y|X)$ can be explicitly computed as $\EX[-\log f_{Y|X}(Y|X)]$. As mentioned before, the aforementioned failure of (\ref{I-H}) can be salvaged with the notion of differential entropy:
\begin{equation} \label{I-h}
I(X; Y)=h(Y)-h(Y|X).
\end{equation}
Note that this equality, together with the fact that the mutual information is non-negative, immediately implies that
\begin{equation} \label{reduces}
h(Y|X) \leq h(Y);
\end{equation}
in other words, conditioning reduces the differential entropy (the same statement holds for entropy, which however is not needed in this paper). We will use the fact that for a given variance, a Gaussian distribution maximizes the differential entropy. More precisely, for a real-valued random variable $Y$ with variance less than or equal to $\sigma^2$, we have
\begin{equation} \label{maximizes}
h(X) \leq h(N) = \frac{1}{2} \log (2 \pi e \sigma^2),
\end{equation}
where $N$ is a Gaussian random variable with mean $0$ and variance $\sigma^2$. 

Here we emphasize that all the above-mentioned definitions naturally carry over to the setting where some/all of random variables are vector-valued. For a quick example, let $Y = \{Y_1, Y_2, \dots, Y_n\}$, where each $Y_i$ is a real-valued random variable with density. Then, the differential entropy $h(Y)$ of $Y$ is defined as
$$
h(Y)=h(Y_1, Y_2, \dots, Y_n) \triangleq \EX[- \log f_{Y_1, Y_2, \dots, Y_n}(Y_1, Y_2, \dots, Y_n)],
$$
where $f_{Y_1, Y_2, \dots, Y_n}$ is the joint probability density function of $Y_1, Y_2, \dots, Y_n$. The {\em chain rule} for differential entropy states that
\begin{equation} \label{chain-rule}
h(Y_1, Y_2, \dots, Y_n) = h(Y_1) + h(Y_2|Y_1)+ \dots+ h(Y_n|Y_1, Y_2, \dots, Y_{n-1}),
\end{equation}
whose conditional version given a random variable $X$ reads
\begin{equation} \label{conditional-chain-rule}
h(Y_1, Y_2, \dots, Y_n|X) = h(Y_1|X) + h(Y_2|X, Y_1)+ \dots+ h(Y_n|X, Y_1, Y_2, \dots, Y_{n-1}).
\end{equation}

Consider the point-to-point continuous-time white Gaussian feedback channel in (\ref{feedback-memory}), the mutual information of which over the time interval $[0, T]$ can be computed as below (see, e.g., ~\cite{pi64,ih93}):
\begin{equation} \label{definition-mutual-information}
I(M; Y_0^T)=\begin{cases}
\EX\left[\log \frac{d \mu_{M, Y_0^T}}{d \mu_{M} \times \mu_{Y_0^T}}(M, Y_0^T)\right], &\mbox{ if } \frac{d \mu_{M, {Y_0^T}}}{d \mu_{M} \times \mu_{Y_0^T}} \mbox{ exists },\\
\infty, &\mbox{ otherwise },
\end{cases}
\end{equation}
where $d \mu_{M, Y_0^T}/d \mu_M \times \mu_{Y_0^T}$ denotes the Radon-Nikodym derivative of $\mu_{M, Y_0^T}$ with respect to $d \mu_M \times \mu_{Y_0^T}$. It turns out that the mutual information is intimately connected to the channel capacity, detailed below.

For $T, R, P> 0$, a $(T, e^{T R}, P)$-{\em code} for the above-mentioned continuous-time Gaussian channel consists of a set of integers $\mathcal{M}=\{1, 2, \ldots, e^{T R}\}$, the {\it message set} for receiver, and an {\it encoding function}, $g: \mathcal{M} \rightarrow C[0, T]$, which satisfies the following average power constraint: with probability $1$,
\begin{equation} \label{PowerConstraint-BC}
\frac{1}{T} \int_0^T \EX[g^2(s, M, Y_0^s)] ds \leq P,
\end{equation}
and a {\it decoding functions}, $h: C[0, T] \rightarrow \mathcal{M}$. The {\em average probability of error} for the $(T, e^{T R}, P)$-code is defined as
$$
P_e^{(T)}=\frac{1}{e^{T R}} \sum_{M \in \mathcal{M}} \mathbb{P}(h(Y_0^T) \neq M~|~M \mbox{ sent}).
$$
A rate $R$ is said to be {\em achievable} for the channel if there exists a sequence of $(T, e^{T R}, P)$-codes with $P_e^{(T)} \rightarrow 0$ as $T \rightarrow \infty$. The {\em capacity} $\mathcal{C}$ of the channel is the supremum of all such achievable rates.

The celebrated Shannon's channel coding theorem~\cite{sh48} states that
$$
\mathcal{C}= \lim_{T \to \infty} \frac{1}{T} \sup_{M, g} I(M; Y_0^T),
$$
where the supremum is over all choices of the message $M$ (including its alphabet and distribution), and the encoding function $g$. As mentioned in Section~\ref{intro}, the capacity of the channel (\ref{feedback-memory}) is $P/2$, the same as that of the channel (\ref{g-formulation}), where the feedback is absent.

\section{Information-Theoretic Extensions} \label{extensions}

In this section, we will establish a sampling and approximation theorem for the channel (\ref{feedback-memory}) restricted to the time interval $[0, T]$:
\begin{equation} \label{feedback-memory-T}
Y(t)=\int_0^t g(s, M, Y_0^{s}) ds + B(t), \quad t \in [0, T],
\end{equation}
which naturally connect such a channel with its discrete-time versions obtained by sampling and approximation.

The following regularity conditions may be imposed to our channel (\ref{feedback-memory-T}):
\begin{itemize}
\item[(a)] The solution $\{Y(t)\}$ to the stochastic differential equation (\ref{feedback-memory-T}) uniquely exists.
\item[(b)]
$$
\PX\left(\int_0^T g^2(t, M, Y_0^t) dt < \infty \right)=\PX\left(\int_0^T g^2(t, M, B_0^t) dt < \infty \right)=1.
$$
\item[(c)]
$$
\int_0^T \EX[|g(t, M, Y_0^t)|] dt < \infty.
$$
\item[(d)] \textbf{The uniform Lipschitz condition:} There exists a constant $L > 0$ such that for any $0 \leq s_1, s_2, t_1, t_2 \leq T$, any $Y_0^T, Z_0^T$,
$$
|g(s_1, M, Y_0^{s_2})-g(t_1, M, Z_0^{t_2})| \leq L (|s_1-t_1|+ \|Y_{0}^{s_2}- Z_0^{t_2}\|).
$$

\item[(e)] \textbf{The uniform linear growth condition:} There exists a constant $L > 0$ such that for any $M$ and any $Y_0^T$,
$$
|g(t, M, Y_0^t)| \leq L (1+\|Y_0^t\|).
$$
\end{itemize}

The following lemma says that Conditions (d)-(e) are stronger than Conditions (a)-(c). We remark that Conditions (d)-(e) are still rather mild assumptions: conditions of similar nature are typically imposed to guarantee the existence and uniqueness of the solution to a given stochastic differential equation, and it is possible that the solution may not uniquely exist if these two conditions are violated (see, e.g.,~\cite{mao97}).
\begin{lem} \label{improved-liptser-1}
Assume Conditions (d)-(e). Then, there exists a unique strong solution of (\ref{feedback-memory}) with initial value $Y(0)=0$. Moreover, there exists $\varepsilon > 0$ such that
\begin{equation} \label{exponential-finiteness-1}
\EX [e^{\varepsilon \|Y_0^T\|^2}] < \infty,
\end{equation}
which immediately implies Conditions (b) and (c).
\end{lem}

\begin{proof}
With Conditions (d)-(e), the proof of the existence and uniqueness of the solution to (\ref{feedback-memory}) is somewhat standard; see, e.g., Section $5.4$ in~\cite{mao97}.
So, in the following, we will only prove (\ref{exponential-finiteness-1}).

For the stochastic differential equation (\ref{feedback-memory}), applying Condition (e), we deduce that there exists $L_1 > 0$ such that
$$
\|Y_0^T\| \leq \int^T_0 L_1(1+ \|Y_0^t\|) dt+ \|B_0^T\| \leq L_1 T + \|B_0^T\| + \int_0^T L_1 \|Y_0^t\| dt.
$$
Then, applying the Gronwall inequality followed by a straightforward bounding analysis, we deduce that there exists $L_2 > 0$ such that
$$
\|Y_0^T\| \leq (L_1 T+ \|B_0^T\|) e^{\int_0^T L_1 dt} = e^{L_1 T} (L_1 T + \|B_0^T\|) \leq L_2 + L_2 \|B_0^T\|.
$$
Now, for any $\varepsilon > 0$, applying Doob's submartingale inequality, we have
\begin{align*}
\EX[e^{\varepsilon \|Y_0^T\|^2}] & \leq \EX[e^{\varepsilon(L_2  + L_2 \|B_0^T\|)^2}] \\
&\leq \EX[e^{2\varepsilon(L_2^2  + L_2^2 \|B_0^T\|^2)}] \\
&=e^{2\varepsilon L_2^2}  \EX[e^{2 \varepsilon L_2^2 \|B_0^T\|^2}]\\
&=e^{2\varepsilon L_2^2}  \EX[\sup\nolimits_{0\leq t\leq T} e^{2 \varepsilon L_2^2 B(t)^2}]\\
&\leq 4 e^{2\varepsilon L_2^2} \EX[e^{2 \varepsilon L_2^2 B(T)^2}],
\end{align*}
which is finite provided that $\varepsilon$ is small enough.

\end{proof}

For any $n \in \mathbb{N}$, choose time points $t^{(n)}_0, t^{(n)}_1, \ldots, t^{(n)}_n \in \mathbb{R}$ such that
$$
0=t^{(n)}_0 < t^{(n)}_1 < \ldots < t^{(n)}_n =T,
$$
and we define
$$
\Delta_n \triangleq \{t^{(n)}_0, t^{(n)}_1, \dots, t^{(n)}_n\}.
$$
For any time point sequence $\Delta_n$, we will use $\delta_{\Delta_n}$ to denote its maximal stepsize, namely,
$$
\delta_{\Delta_n} \triangleq \max_{i=1, 2, \dots, n} (t^{(n)}_i-t^{(n)}_{i-1}).
$$
$\Delta_n$ is said to be {\it evenly spaced} if $t^{(n)}_i-t^{(n)}_{i-1}=T/n$ for all feasible $i$, and we will use the shorthand notation $\delta_n$ to denote its stepsize, i.e., $\delta_n \triangleq t^{(n)}_1-t^{(n)}_0=T/n$. Apparently, evenly spaced sampling time sequences are natural candidates with respect to which a continuous-time Gaussian channel can be sampled.

In the following, we will present our extensions of the Shannon-Nyquist sampling theorem vis a sampling or approximation with respect to $\Delta_n$.

\subsection{Extension by Sampling} \label{sampling-theorems}

In this section, we will establish a sampling theorem for the channel (\ref{feedback-memory}), which naturally connect such a channel with their discrete-time versions obtained by sampling.

Sampling the channel (\ref{feedback-memory}) over the time interval $[0, T]$ with respect to $\Delta_n$, we obtain its sampled discrete-time version as follows:
\begin{equation}  \label{after-sampling-with-feedback}
Y(t^{(n)}_i)=\int_0^{t^{(n)}_i} g(s, M, Y_0^s) ds + B(t^{(n)}_i), \quad i=0, 1, \ldots, n.
\end{equation}
Roughly speaking, the following sampling theorem states that for any sequence of ``increasingly finer'' samplings, the mutual information of the sampled discrete-time channel (\ref{after-sampling-with-feedback}) will converge to that of the original channel (\ref{feedback-memory-T}).
\begin{thm} \label{sampling-theorem}
1) Assume Conditions (a)-(c). Suppose that $\Delta_n \subset \Delta_{n+1}$ for all $n$ and that $\delta_{\Delta_n} \to 0$ as $n$ tends to infinity. Then, we have
$$
\lim_{n \to \infty} I(M; Y(\Delta_n))=I(M; Y_0^T),
$$
where we recall from Section~\ref{notations} that $Y(\Delta_n) = \{Y(t^{(n)}_0), Y(t^{(n)}_1), \ldots, Y(t^{(n)}_n)\}$.

2) Assume Conditions (d)-(e). Suppose that $\{\Delta_n\}$ with $\delta_{\Delta_n} \to 0$ as $n$ tends to infinity. Then, we have
$$
\lim_{n \to \infty} I(M; Y(\Delta_n))=I(M; Y_0^T).
$$
\end{thm}

\begin{proof}
The proof is rather technical and lengthy, and thereby postponed to Section~\ref{proof-sampling-theorem}.
\end{proof}

\begin{rem}
Regarding the assumptions of Theorem~\ref{sampling-theorem}, Conditions (a)-(c) in 1) are rather weak, but the condition that ``$\Delta_n \subset \Delta_{n+1}$ for all $n$'' is somewhat restrictive, which, in particular, is not satisfied by the set $\{\Delta_n\}$ of all evenly spaced time point sequences. As shown in 2), this condition can be removed at the expenses of the extra regularity conditions (d)-(e): The same theorem holds as long as the stepsize of the sampling tends to $0$, which, in particular, is satisfied by the set of all evenly spaced time point sequences.
\end{rem}

\subsection{Extension by Approximation} \label{approximation-theorems}

In this section, we will establish an approximation theorem for the channel (\ref{feedback-memory-T}), which naturally connect such a channel with their discrete-time versions obtained by approximation.

An application of the Euler-Maruyama approximation~\cite{kl92} with respect to $\Delta_n$ to (\ref{feedback-memory-T}) will yield a discrete-time sequence $\{Y^{(n)}(t^{(n)}_i): i=0, 1, \dots, n\}$ and a continuous-time process $\{Y^{(n)}(t): t \in [0, T]\}$, a linear interpolation of $\{Y(t^{(n)}_i)\}$, as follows: Initializing with $Y^{(n)}(0)=0$, we recursively compute, for each $i=0, 1, \dots, n-1$,
\begin{equation} \label{Euler-Maruyama-Sequence}
Y^{(n)}(t^{(n)}_{i+1})=Y^{(n)}(t^{(n)}_i)+ \int_{t^{(n)}_i}^{t^{(n)}_{i+1}} g(s, M, Y_0^{(n), t^{(n)}_i}) ds + B(t^{(n)}_{i+1})-B(t^{(n)}_i),
\end{equation}
\begin{equation} \label{linear-interpolation}
Y^{(n)}(t)=Y^{(n)}(t^{(n)}_i)+\frac{t-t^{(n)}_i}{t^{(n)}_{i+1}-t^{(n)}_i} (Y^{(n)}(t^{(n)}_{i+1})-Y^{(n)}(t^{(n)}_i)), \quad t^{(n)}_i \leq t \leq t^{(n)}_{i+1}.
\end{equation}

We will first show the following lemma, which is parallel to Lemma~\ref{improved-liptser-1}.
\begin{lem} \label{improved-liptser-2}
Assume Conditions (d)-(e). Then, there exists $\varepsilon > 0$ and a constant $C > 0$ such that for all $n$,
\begin{equation} \label{exponential-finiteness-2}
\EX [e^{\varepsilon \|Y_0^{(n), T}\|^2}] < C.
\end{equation}
\end{lem}

\begin{proof}
A discrete-time version of the proof of Lemma~\ref{improved-liptser-1} implies that there exists $\varepsilon > 0$ and a constant $C > 0$ such that for all $n$
$$
\EX [e^{\varepsilon \sup_{i \in \{0, 1, \dots, n\}} (Y^{(n)}(t^{(n)}_i))^2}] < C,
$$
which, together with (\ref{linear-interpolation}), immediately implies (\ref{exponential-finiteness-2}).
\end{proof}

We also need the following lemma, which is parallel to Theorem $10.2.2$ in~\cite{kl92}.
\begin{lem}  \label{Y-Y}
Assume Conditions (d)-(e). Then, there exists a constant $C > 0$ such that for all $n$,
$$
\EX[\|Y_0^{(n), T}-Y_0^T\|^2] \leq C \delta_{\Delta_n}.
$$
\end{lem}

\begin{proof}
Note that for any $n$, we have
$$
Y(t^{(n)}_{i+1})=Y(t^{(n)}_i)+\int_{t^{(n)}_i}^{t^{(n)}_{i+1}} g(s, M, Y_0^s) ds + B(t^{(n)}_{i+1})-B(t^{(n)}_i),
$$
and
$$
Y^{(n)}(t^{(n)}_{i+1})=Y^{(n)}(t^{(n)}_{i})+\int_{t^{(n)}_{i}}^{t^{(n)}_{i+1}} g(s, M, Y_0^{(n),t^{(n)}_{i}}) ds+ B(t^{(n)}_{i+1})-B(t^{(n)}_{i}).
$$
It then follows that
\begin{equation} \label{Y-diff}
Y(t^{(n)}_{i+1})-Y^{(n)}(t^{(n)}_{i+1})=Y(t^{(n)}_i)-Y^{(n)}(t^{(n)}_i)+\int_{t^{(n)}_i}^{t^{(n)}_{i+1}} (g(s, M, Y_0^s)-g(s, M, Y_0^{(n), t^{(n)}_i})) ds.
\end{equation}
Now, for any $t$, choose $n_0$ such that $t_{n_0}^{(n)} \leq t < t_{n_0+1}^{(n)}$. Now, a recursive application of (\ref{Y-diff}), coupled with Conditions (d) and (e), yields that for some $L > 0$,
{\small \begin{align*}
\hspace{-1.5cm} Y(t)-Y^{(n)}(t)&=\sum_{i=0}^{n_0} \int_{t^{(n)}_{i}}^{t^{(n)}_{i+1}} (g(s, M, Y_0^s)-g(t^{(n)}_{i}, M, Y_0^{(n),t^{(n)}_{i}})) ds+\int_{t_{n, n_0+1}}^{t} (g(s, M, Y_0^s)-g(t^{(n)}_{i}, M, Y_0^{(n),t_{n, n_0+1}})) ds\\
&\leq \sum_{i=0}^{n_0} \int_{t^{(n)}_{i}}^{t^{(n)}_{i+1}} L |s-t^{(n)}_i| + L \|Y_0^s-Y_0^{(n),s}\|+L \|Y_0^{(n),s}-Y_0^{(n),t^{(n)}_{i}}\| ds\\
&+\int_{t_{n_0+1}^{(n)}}^{t} L |s-t_{n_0+1}^{(n)}| + L \|Y_0^s-Y_0^{(n),s}\|+L \|Y_0^{(n),s}-Y_0^{(n),t_{n_0+1}^{(n)}}\| ds.
\end{align*}}
Noticing that for any $s$ with $t^{(n)}_i \leq s < t^{(n)}_{i+1}$, we have
$$
\hspace{-1cm} \|Y_0^{(n),s}-Y_0^{(n),t^{(n)}_{i}}\|^2 \leq |Y^{(n)}(t^{(n)}_{i+1})-Y^{(n)}(t^{(n)}_{i})|^2 \leq 2 \left| \int_{t^{(n)}_{i}}^{t^{(n)}_{i+1}} g(s, M, Y_0^{(n),t^{(n)}_{i}}) ds \right|^2+ 2 |B(t^{(n)}_{i+1})-B(t^{(n)}_{i})|^2,
$$
which, together with Condition (e) and the fact that for all $n$ and $i$,
\begin{equation} \label{BB}
\EX[|B(t^{(n)}_{i+1})-B(t^{(n)}_{i})|^2] = O(\delta_{\Delta_n}),
\end{equation}
implies that
\begin{equation} \label{YY}
\EX[\|Y_0^{(n),s}-Y_0^{(n),t^{(n)}_{i}}\|^2] = O(\delta_{\Delta_n}).
\end{equation}
Noting that the constants in the two terms $O(\delta_{\Delta_n})$ in (\ref{BB}) and (\ref{YY}) can be chosen uniform over all $n$, a usual argument with the Gronwall inequality applied to $\EX[\|Y_0^t-Y_0^{(n), t}\|^2]$ completes the proof of the lemma.
\end{proof}

We are now ready to state and prove the following theorem:
\begin{thm} \label{approximation-theorem}
Assume Conditions (d)-(e). Then, we have
$$
\lim_{n \to \infty} I(M; Y^{(n)}(\Delta_n))=I(M; Y_0^T).
$$
\end{thm}

\begin{proof}
The proof is rather technical and lengthy, and thereby postponed to Section~\ref{proof-approximation-theorem}.
\end{proof}

\begin{rem} \label{sampling-approximation-theorems}
When there is no feedback, Theorem~\ref{approximation-theorem} boils down to Theorem~\ref{sampling-theorem}: obviously it holds that for any feasible $i$,
$$
Y^{(n)}(t^{(n)}_i)=Y(t^{(n)}_i),
$$
which means that Theorem~\ref{approximation-theorem} actually states
$$
\lim_{n \to \infty} I(M; Y(\Delta_n))=I(M; Y_0^T),
$$
which is precisely the conclusion of Theorem~\ref{sampling-theorem}.
\end{rem}

\section{Applications of Our Results} \label{multi-user-Gaussian-channels}

In this section, we discuss possible applications of our extensions. Evidently, establishing causality-preserving connections between continuous-time and discrete-time Gaussian feedback channels, our results may help channel the ideas and techniques in the discrete-time regime to the continuous-time one. Below, we use an example to illustrate this point.

Consider the continuous-time white Gaussian feedback channel as in (\ref{feedback-memory}) and assume that Conditions (d) and (e) are satisfied, and moreover the following average power constraint is satisfied: there exists $P > 0$ such that for any $T$,
\begin{equation}  \label{power}
\frac{1}{T} \int_0^T \EX[g^2(s, M, Y_0^s)] ds \leq P.
\end{equation}
It has been established in~\cite{ka71} that
$$
I(M; Y_0^T) = \frac{1}{2} \int_0^{T} \EX[g^2(s, M, Y_0^s)]-\EX[\EX^2[g(s, M, Y_0^s)|Y_0^s])] ds,
$$
which, together with (\ref{power}), immediately implies that
\begin{equation} \label{half-PT}
I(M; Y_0^T) \leq \frac{PT}{2}.
\end{equation}
Below, we will use Theorem~\ref{approximation-theorem} and some basic facts for discrete-time Gaussian channels to derive (\ref{half-PT}), which, in combination with the proven fact that one can choose $g$, $M$ and sufficiently large $T$ so that $I(M; Y_0^T)/T$ is arbitrarily close to $P/2$ (see Theorem $6.4.1$ in~\cite{ih93}), implies that the capacity $\mathcal{C}$ of the channel (\ref{feedback-memory}) is indeed $P/2$.

First of all, for fixed $T > 0$, consider the evenly spaced $\Delta_n$ with stepsize $\delta_n = T/n$. Applying the Euler-Maruyama approximation (\ref{Euler-Maruyama-Sequence}) to the channel (\ref{feedback-memory-T}) over the time window $[0, T]$, we obtain
\begin{equation} \label{revisited}
Y^{(n)}(t^{(n)}_{i+1})=Y^{(n)}(t^{(n)}_i)+ \int_{t^{(n)}_i}^{t^{(n)}_{i+1}} g(s, M, Y_0^{(n), t^{(n)}_i}) ds + B(t^{(n)}_{i+1})-B(t^{(n)}_i).
\end{equation}
By Theorem~\ref{approximation-theorem}, we have
\begin{equation}  \label{another-heuristic}
I(M; Y_0^T) = \lim_{n \to \infty} I(M; Y^{(n)}(\Delta_n)).
\end{equation}
Note that, it can be easily verified that
\begin{align*}
h(\{Y^{(n)}(t^{(n)}_{i})-Y^{(n)}(t^{(n)}_{i-1})\}_{i=1}^n|M) & \stackrel{(a)}{=} \sum_{i=1}^{n} h(Y^{(n)}(t^{(n)}_i)-Y^{(n)}(t^{(n)}_{i-1})|\{Y^{(n)}(t^{(n)}_{j})-Y^{(n)}(t^{(n)}_{j-1})\}_{j=1}^{i-1}, M)\\
& \stackrel{(b)}{=}h(\{B(t^{(n)}_{i})-B(t^{(n)}_{i-1})\}_{i=1}^n|M)\\
& \stackrel{(c)}{=}h(\{B(t^{(n)}_{i})-B(t^{(n)}_{i-1})\}_{i=1}^n)\\
& \stackrel{(d)}{=}\sum_{i=1}^{n} h(B(t^{(n)}_i)-B(t^{(n)}_{i-1})),
\end{align*}
where we have used the conditional chain rule for differential entropy for (a); and (\ref{revisited}) for (b); and the independence between $\{B(t)\}$ and $M$ for (c); and (\ref{chain-rule}) and the fact that a Brownian motion has independent increments for (d). Then, using (\ref{I=I}), (\ref{I-h}), (\ref{chain-rule}) and (\ref{reduces}), we find that
\begin{align*}
I(M; Y^{(n)}(\Delta_n)) &= I(M; \{Y^{(n)}(t^{(n)}_{i})-Y^{(n)}(t^{(n)}_{i-1})\}_{i=1}^n\})\\
&= h(\{Y^{(n)}(t^{(n)}_{i})-Y^{(n)}(t^{(n)}_{i-1})\}_{i=1}^n)-h(\{Y^{(n)}(t^{(n)}_{i})-Y^{(n)}(t^{(n)}_{i-1})\}_{i=1}^n|M)\\
&=\sum_{i=1}^{n} h(Y^{(n)}(t^{(n)}_i)-Y^{(n)}(t^{(n)}_{i-1})|\{Y^{(n)}(t^{(n)}_{j})-Y^{(n)}(t^{(n)}_{j-1})\}_{j=1}^{i-1})-\sum_{i=1}^{n} h(B(t^{(n)}_i)-B(t^{(n)}_{i-1}))\\
& \leq \sum_{i=1}^{n} h(Y^{(n)}(t^{(n)}_i)-Y^{(n)}(t^{(n)}_{i-1}))-\sum_{i=1}^{n} h(B(t^{(n)}_i)-B(t^{(n)}_{i-1})).
\end{align*}
Next, using the Cauchy-Schwarz inequality, we reach
\begin{align*}
Var(Y^{(n)}(t^{(n)}_i)-Y^{(n)}(t^{(n)}_{i-1})) &= \EX[(Y^{(n)}(t^{(n)}_i)-Y^{(n)}(t^{(n)}_{i-1}))^2]\\
& \leq \delta_n \EX\left[\int_{t^{(n)}_{i-1}}^{t^{(n)}_i} g^2(s, M, Y_0^{(n), t^{(n)}_{i-1}}) ds \right]+\EX[(B(t^{(n)}_i)-B(t^{(n)}_{i-1}))^2]\\
&=\delta_n \int_{t^{(n)}_{i-1}}^{t^{(n)}_i} \EX\left[g^2(s, M, Y_0^{(n), t^{(n)}_{i-1}}) \right] ds +\delta_n,
\end{align*}
which, together with (\ref{maximizes}), further implies that
\begin{align}
\label{ineq-1} I(M; Y^{(n)}(\Delta_n)) & \leq \frac{1}{2} \sum_{i=0}^{n} \log \left(1+\int_{t^{(n)}_{i-1}}^{t^{(n)}_i} \EX\left[g^2(s, M, Y_0^{(n), t^{(n)}_{i-1}}) \right] ds \right) \\
\label{ineq-2} & \leq \frac{1}{2} \sum_{i=0}^{n} \int_{t^{(n)}_{i-1}}^{t^{(n)}_i} \EX\left[g^2(s, M, Y_0^{(n), t^{(n)}_{i-1}})\right] ds.
\end{align}
It then follows from (\ref{another-heuristic}), Condition (d), Lemma~\ref{Y-Y} and (\ref{power}) that
\begin{equation} \label{ineq-3}
I(M; Y_0^T) \leq \frac{1}{2} \int_0^T \EX[g^2(s, M, Y_0^s)] ds \leq \frac{P T}{2},
\end{equation}
as desired.

\section{Proof of Theorem~\ref{sampling-theorem}}  \label{proof-sampling-theorem}

First of all, recall from Section~\ref{notations} that for a stochastic process $\{X(s)\}$ and any $t \in \mathbb{R}_+$, we use $\mu_{X_0^t}$ to denote the distribution on $C[0, t]$ induced by $X_0^t$. Throughout the proof, we only have to deal with the case $t=T$, and so we will simply write $\mu_{X_0^T}$ as $\mu_X$ for notational simplicity. And for cosmetic reasons, we will write the summation $\sum_m (\cdot) p_M(m)$ as the integral $\int (\cdot) d\mu_M(m)$.

\subsection{Proof of 1)}

Note that an application of Theorem $7.14$ of~\cite{li01} with Conditions (b) and (c) yields that
\begin{equation} \label{crazy}
P\left(\int_0^T \EX^2[g(t, M, Y_0^t)|Y_0^t] dt < \infty \right)=1.
\end{equation}
Then one verifies that the assumptions of Lemma $7.7$ of~\cite{li01} are all satisfied (this lemma is stated under very general assumptions, which are exactly Conditions (b), (c) and (\ref{crazy}) when restricted to our settings), which implies that for any $m$,
\begin{equation}  \label{two-tilde}
\mu_Y \sim \mu_{Y|M=m} \sim \mu_B,
\end{equation}
and moreover, with probability $1$,
{\small \begin{equation}  \label{RN-1}
\frac{d\mu_{Y|M}}{d\mu_B}(Y_0^T|M)=\frac{1}{\EX[e^{A_1(M, Y_0^T)}|Y_0^T, M]}, \quad \frac{d\mu_{Y}}{d\mu_B}(Y_0^T)=\frac{1}{\EX[e^{A_1(M, Y_0^T)}|Y_0^T]},
\end{equation}}
where
$$
A_1(m, Y_0^T) \triangleq -\int_0^T g(s, m, Y_0^s) dY(s)+\frac{1}{2} \int_0^T g(s, m, Y_0^s)^2 ds,
$$
and
$$
A_1(M, Y_0^T) \triangleq -\int_0^T g(s, M, Y_0^s) dY(s)+\frac{1}{2} \int_0^T g(s, M, Y_0^s)^2 ds.
$$
Here we note that $\EX[e^{A_1(M, Y_0^T)}|Y_0^T, M]$ is in fact equal to $e^{A_1(M, Y_0^T)}$, but we keep it the way it is as above for an easy comparison.

Note that it follows from $\EX[d\mu_B/d\mu_Y(Y_0^T)]=1$ that $\EX[e^{A_1(M, Y_0^T)}]=1$, which is equivalent to
\begin{equation}  \label{weak-Novikov}
\EX[e^{-\int_0^T g(s, M, Y_0^s) dB(s)-\frac{1}{2} \int_0^T g(s, M, Y_0^s)^2 ds}]=1.
\end{equation}
Then, a parallel argument as in the proof of Theorem $7.1$ of~\cite{li01} (which requires the condition (\ref{weak-Novikov})) further implies that, for any $n$,
{\small \begin{equation}  \label{RN-2}
\hspace{-1.0cm} \frac{d\mu_{Y(\Delta_n)|M}}{d\mu_{B(\Delta_n)}}(Y(\Delta_n)|M) = \frac{1}{\EX[e^{A_1(M, Y_0^T)}|Y(\Delta_n), M]}, \quad \frac{d\mu_{Y(\Delta_n)}}{d\mu_{B(\Delta_n)}}(Y(\Delta_n))=\frac{1}{\EX[e^{A_1(M, Y_0^T)}|Y(\Delta_n)]}, ~~\mbox{a.s.}.
\end{equation}}
Notice that it can be easily checked that $e^{A_1(M, Y_0^T)}$ integrable, which, together with the assumption that $\Delta_n \subset \Delta_{n+1}$ for all $n$, further implies that
$$
\{\EX[e^{A_1(M, Y_0^T)}|Y(\Delta_n), M]\}, \quad \{\EX[e^{A_1(M, Y_0^T)}|Y(\Delta_n)]\}
$$
are both martingales, and therefore, by Doob's martingale convergence theorem~\cite{Durrett},
$$
\frac{d\mu_{Y(\Delta_n)|M}}{d\mu_{B(\Delta_n)}}(Y(\Delta_n)|M) \to \frac{d\mu_{Y|M}}{d\mu_B}(Y_0^T|M), \quad \frac{d\mu_{Y(\Delta_n)}}{d\mu_{B(\Delta_n)}}(Y(\Delta_n)) \to \frac{d\mu_{Y}}{d\mu_B}(Y_0^T), \mbox{ a.s.}.
$$

Now, by Jensen's inequality, we have
\begin{equation} \label{xianming-inequality-1}
\EX\left[ \left. A_1(M, Y_0^T) \right| Y(\Delta_n), M \right] \leq \log\EX[e^{A_1(M, Y_0^T)}|Y(\Delta_n), M],
\end{equation}
and, by the fact that $\log x \leq x$ for any $x > 0$, we have
\begin{equation} \label{xianming-inequality-2}
\log\EX[e^{A_1(M, Y_0^T)}|Y(\Delta_n), M] \leq \EX[e^{A_1(M, Y_0^T)}|Y(\Delta_n), M].
\end{equation}
It then follows from (\ref{xianming-inequality-1}) and (\ref{xianming-inequality-2}) that
\begin{align*}
\left|\log\EX[e^{A_1(M, Y_0^T)}|Y(\Delta_n), M]\right| \leq \left|\EX\left[A_1(M, Y_0^T)|Y(\Delta_n), M \right]\right| +\EX[e^{A_1(M, Y_0^T)}|Y(\Delta_n), M].
\end{align*}
Applying the general Lebesgue dominated convergence theorem (see, e.g., Theorem $19$ on Page $89$ of~\cite{ro10}), we then have
\begin{equation} \label{l-1}
\lim_{n \to \infty} \EX\left[\log \frac{d\mu_{Y|M}}{d\mu_B}(Y(\Delta_n)|M)\right] = \EX [\log  \EX[e^{A_1(M, Y_0^T)}|Y^{T}_{0}, M]] = \EX\left[\log \frac{d\mu_{Y|M}}{d\mu_B}(Y_{0}^{T}|M)\right].
\end{equation}
A completely parallel argument yields that
\begin{equation} \label{l-2}
\lim_{n \to \infty} \EX\left[\log \frac{d\mu_{Y(\Delta_n)}}{d\mu_{B(\Delta_n)}}(Y(\Delta_n)) \right] = \EX[\log  \EX[\exp(A_1(M, Y_0^T))|Y^{T}_{0}]] = \EX\left[\log \frac{d\mu_{Y}}{d\mu_B}(Y_{0}^{T})\right].
\end{equation}
Next, by the definition of mutual information, we have
\begin{align}
I(M; Y(\Delta_n)) &= \EX\left[\log f(\mu_{Y(\Delta_n)|M})(Y(\Delta_n)|M)\right]-\EX\left[\log f(\mu_{Y(\Delta_n)})(Y(\Delta_n))\right] \nonumber\\
& = \EX\left[\log \frac{d\mu_{Y(\Delta_n)|M}}{d\mu_{B(\Delta_n)}}(Y(\Delta_n)|M)\right]-\EX\left[\log \frac{d\mu_{Y(\Delta_n)}}{d\mu_{B(\Delta_n)}}(Y(\Delta_n))\right], \label{l-3}
\end{align}
and
\begin{align}
I(M; Y_0^T) & = \EX\left[\log \frac{d \mu_{M, Y_0^T}}{d \mu_{M} \times \mu_{Y_0^T}}(M, Y_0^T)\right] \nonumber\\
& =\EX\left[\log \frac{d\mu_{Y|M}}{d\mu_{B}}(Y_0^T|M)\right]-\EX\left[\log \frac{d\mu_{Y}}{d\mu_B}(Y_0^T)\right], \label{l-4}
\end{align}
where the well-definedness of the Radon-Nikodym derivatives are guaranteed by (\ref{two-tilde}).

Finally, with (\ref{l-1}), (\ref{l-2}), (\ref{l-3}) and (\ref{l-4}), we conclude that
$$
\lim_{n \to \infty} I(M; Y(\Delta_n))=\EX\left[\log \frac{d\mu_{Y|M}}{d\mu_B}(Y_0^T|M)\right]-\EX\left[\log \frac{d\mu_{Y}}{d\mu_B}(Y_0^T)\right]=I(M;Y_0^T),
$$
as desired.

\subsection{Proof of 2)}
We proceed in the following steps.

{\bf Step $\bf 1$.} In this step, we establish the theorem assuming that there exists $C > 0$ such that for all $m \in \mathcal{M}$ and all $y_0^T \in C[0, T]$,
\begin{equation} \label{case-1}
\int_0^T g^2(s, m, y_0^s) ds < C.
\end{equation}
By (\ref{l-3}), (\ref{RN-1}) and (\ref{RN-2}), we have
{\small \begin{align*}
I(M; Y(\Delta_n)) &= \EX\left[\log \frac{d\mu_{Y(\Delta_n)|M}}{d\mu_{B(\Delta_n)}}(Y(\Delta_n)|M)\right]-\EX\left[\log \frac{d\mu_{Y(\Delta_n)}}{d\mu_{B(\Delta_n)}}(Y(\Delta_n))\right]\\
&= -\EX[\log \EX[e^{A_1(M, Y_0^T)}|Y(\Delta_n), M]] + \EX[\log \EX[e^{A_1(M, Y_0^T)}|Y(\Delta_n)]]
\end{align*}

{\bf Step $\bf 1.1$.} In this step, we prove that as $n$ tends to infinity,
\begin{equation} \label{Fn-Conv}
\log \EX[e^{A_1(M, Y_0^T)}|Y(\Delta_n), M] \to A_1(M, Y_0^T),
\end{equation}
in probability.

Let $\bar{Y}_{\Delta_n, 0}^T$ denote the piecewise linear version of $Y_0^T$ with respect to $\Delta_n$; more precisely, for any $i=0, 1, \dots, n$, let $\bar{Y}_{\Delta_n}(t^{(n)}_{i})=Y(t^{(n)}_{i})$, and for any $t^{(n)}_{i-1} < s < t^{(n)}_i$ with $s=\lambda t^{(n)}_{i-1}+(1-\lambda) t^{(n)}_i$ for some $0 < \lambda < 1$, let $\bar{Y}_{\Delta_n}(s)=\lambda Y(t^{(n)}_{i-1})+(1-\lambda) Y(t^{(n)}_i)$. Let $\bar{g}_{\Delta_n}(s, M, \bar{Y}_{\Delta_n, 0}^s)$ denote the piecewise ``flat'' version of $g(s, M, \bar{Y}_{\Delta_n, 0}^s)$ with respect to $\Delta_n$; more precisely, for any $t^{(n)}_{i-1} \leq s < t^{(n)}_i$, $\bar{g}_{\Delta_n}(s, M, \bar{Y}_{\Delta_n, 0}^s)=g(t^{(n)}_{i-1}, M, \bar{Y}_{\Delta_n, 0}^{t^{(n)}_{i-1}})$.

Letting
$$
A_2(\Delta_n, m, Y_0^T) \triangleq -\int_0^T \bar{g}_{\Delta_n}(s, m, \bar{Y}_{\Delta_n, 0}^s) dY(s)+\frac{1}{2} \int_0^T \bar{g}^2_{\Delta_n}(s, m, \bar{Y}_{\Delta_n, 0}^s) ds,
$$
and
$$
A_2(\Delta_n, M, Y_0^T) \triangleq -\int_0^T \bar{g}_{\Delta_n}(s, M, \bar{Y}_{\Delta_n, 0}^s) dY(s)+\frac{1}{2} \int_0^T \bar{g}^2_{\Delta_n}(s, M, \bar{Y}_{\Delta_n, 0}^s) ds,
$$
we have
\begin{align*}
\log \EX[e^{A_1(M, Y_0^T)}|Y(\Delta_n), M] & = \log \EX[e^{A_2(\Delta_n, M, Y_0^T) +A_1(M, Y_0^T)-A_2(\Delta_n, M, Y_0^T)}|Y(\Delta_n), M]\\
    & = \log e^{A_2(\Delta_n, M, Y_0^T)} \EX[e^{A_1(M, Y_0^T)-A_2(\Delta_n, M, Y_0^T)}|Y(\Delta_n), M]\\
    & = A_2(\Delta_n, M, Y_0^T) +\log \EX[e^{A_1(M, Y_0^T)-A_2(\Delta_n, M, Y_0^T)}|Y(\Delta_n), M],
\end{align*}
where we have used the fact that
$$
\EX[e^{A_2(\Delta_n, M, Y_0^T)}|Y(\Delta_n), M]=e^{A_2(\Delta_n, M, Y_0^T)},
$$
since $A_2(\Delta_n, M, Y_0^T)$ only depends on $M$ and $Y(\Delta_n)$.

We now prove the following convergence:
\begin{equation} \label{conv-1}
\lim_{n \to \infty} \EX\left[(A_1(M, Y_0^T)-A_2(\Delta_n, M, Y_0^T))^2\right] = 0,
\end{equation}
which will imply that
$$
A_2(\Delta_n, M, Y_0^T) \to A_1(M, Y_0^T)
$$
in probability. To prove (\ref{conv-1}), we note that
$$
A_1(M, Y_0^T)-A_2(\Delta_n, M, Y_0^T) = -\int_0^T (g(s)-\bar{g}_{\Delta_n}(s)) dB(s) - \frac{1}{2} \int_0^T (g(s)-\bar{g}_{\Delta_n}(s))^2 ds,
$$
where we have rewritten $g(s, M, Y_0^s)$ as $g(s)$, and $\bar{g}_{\Delta_n}(s, M, \bar{Y}_{\Delta_n, 0}^s)$ as $\bar{g}_{\Delta_n}(s)$. It then follows that (\ref{conv-1}) boils down to
\begin{equation} \label{conv-1a}
\lim_{n \to \infty} \EX\left[\left(-\int_0^T (g(s)-\bar{g}_{\Delta_n}(s)) dB(s) - \frac{1}{2} \int_0^T (g(s)-\bar{g}_{\Delta_n}(s))^2 ds \right)^2 \right] = 0.
\end{equation}
To establish (\ref{conv-1a}), notice that, by the It\^{o} isometry~\cite{ok95}, we have
$$
\EX\left[\left(\int_0^T (g(s)-\bar{g}_{\Delta_n}(s)) dB(s) \right)^2\right] = \EX\left[\int_0^T (g(s)-\bar{g}_{\Delta_n}(s))^2 ds \right],
$$
which means we only need to prove that
\begin{equation} \label{conv-1b}
\lim_{n \to \infty} \EX\left[\left(\int_0^T (g(s)-\bar{g}_{\Delta_n}(s))^2 ds \right)^2\right] = 0.
\end{equation}
To see this, we note that, by Conditions (d) and (e), there exists $L_1 > 0$ such that for any $s \in [0, T]$ with $t^{(n)}_{i-1} \leq s < t^{(n)}_i$,
\begin{align}
\hspace{-1cm} |g(s, M, \bar{Y}_{\Delta_n, 0}^s)-\bar{g}_{\Delta_n}(s, M, \bar{Y}_{\Delta_n, 0}^s)| & = |g(s, M, \bar{Y}_{\Delta_n, 0}^s)-g(t^{(n)}_{i-1}, M, \bar{Y}_{\Delta_n, 0}^{t^{(n)}_{i-1}})| \nonumber\\
& \leq L_1 (|s-t^{(n)}_{i-1}| +\|\bar{Y}_{\Delta_n, 0}^s-\bar{Y}_{\Delta_n, 0}^{t^{(n)}_{i-1}}\|) \nonumber\\
& \leq L_1 (|s-t^{(n)}_{i-1}| + |Y(t^{(n)}_i)-Y(t^{(n)}_{i-1})|) \nonumber\\
& \leq  L_1 \delta_{\Delta_n}+ L_1 \delta_{\Delta_n}+L_1 \delta_{\Delta_n} \|Y_0^T\|+|B(t^{(n)}_i)-B(t^{(n)}_{i-1})|. \label{diff-2}
\end{align}
Moreover, by Lemma~\ref{improved-liptser-1}, $\|Y_0^T\|^4$ is integrable. And furthermore, one easily verifies that that
\begin{equation} \label{conv-1d}
\EX[(B(t^{(n)}_i)-B(t^{(n)}_{i-1}))^4] = 3 (t^{(n)}_i-t^{(n)}_{i-1})^2 \leq 3 \delta^2_{\Delta_n}.
\end{equation}
It can be readily checked that (\ref{diff-2}) and (\ref{conv-1d}) imply (\ref{conv-1b}), which in turn implies (\ref{conv-1}), as desired.

We now prove the following convergence:
\begin{equation} \label{conv-2}
\lim_{n \to \infty} \EX[|\EX[e^{A_1(M, Y_0^T)-A_2(\Delta_n, M, Y_0^T)}|Y(\Delta_n), M]-1|] = 0,
\end{equation}
which will imply that
$$
\log \EX[e^{A_1(M, Y_0^T)-A_2(\Delta_n, M, Y_0^T)}|Y(\Delta_n), M] \to 0
$$
in probability and furthermore (\ref{Fn-Conv}). To establish (\ref{conv-2}), we first note that
\begin{align*}
&\hspace{-1cm} \EX[|\EX[e^{A_1(M, Y_0^T)-A_2(\Delta_n, M, Y_0^T)}|Y(\Delta_n), M]-1|]\\
&\leq \EX[\EX[|e^{A_1(M, Y_0^T)-A_2(\Delta_n, M, Y_0^T)}-1||Y(\Delta_n), M]]\\
&= \EX[|e^{A_1(M, Y_0^T)-A_2(\Delta_n, M, Y_0^T)}-1|]\\
&\leq \EX\left[\left|A_1(M, Y_0^T)-A_2(\Delta_n, M, Y_0^T)\right| e^{|A_1(M, Y_0^T)-A_2(\Delta_n, M, Y_0^T)|}\right]\\
&\leq \EX\left[\left|A_1(M, Y_0^T)-A_2(\Delta_n, M, Y_0^T)\right|^2\right] \EX\left[e^{2 |A_1(M, Y_0^T)-A_2(\Delta_n, M, Y_0^T)|}\right].
\end{align*}
It then follows from (\ref{conv-1}) that, to prove (\ref{conv-2}), we only need to prove that if $\delta_{\Delta_n}$ is small enough,
\begin{equation} \label{conv-2a}
\EX\left[e^{2 |A_1(M, Y_0^T)-A_2(\Delta_n, M, Y_0^T)|}\right] < \infty.
\end{equation}
Since
\begin{equation} \label{two-terms}
\EX\left[e^{2 |A_1(M, Y_0^T)-A_2(\Delta_n, M, Y_0^T)|}\right] \leq \EX\left[e^{2 (A_1(M, Y_0^T)-A_2(\Delta_n, M, Y_0^T))}\right]+ \EX\left[e^{-2(A_1(M, Y_0^T)-A_2(\Delta_n, M, Y_0^T))}\right],
\end{equation}
we only have to prove that the two terms in the above upper bound are both finite provided that $\delta_{\Delta_n}$ is small enough. Note that for the first term, applying the Cauchy-Schwarz inequality, we have
\begin{align*}
\hspace{-1cm} \EX[e^{2 (A_1(M, Y_0^T)-A_2(\Delta_n, M, Y_0^T))}] & = \EX[e^{ \int_0^T 2 (g(s)-\bar{g}_{\Delta_n}(s)) dB(s) - \int_0^T 4 (g(s)-\bar{g}_{\Delta_n}(s))^2 ds+3 \int_0^T (g(s)-\bar{g}_{\Delta_n}(s))^2 ds}]\\
& \leq \EX[e^{ \int_0^T 4 (g(s)-\bar{g}_{\Delta_n}(s)) dB(s) - \int_0^T 8 (g(s)-\bar{g}_{\Delta_n}(s))^2 ds}] \EX[e^{6 \int_0^T (g(s)-\bar{g}_{\Delta_n}(s))^2 ds}].
\end{align*}
Then, an application of Fatou's lemma yields that
\begin{equation}  \label{Fatou-Lemma}
\EX[e^{ \int_0^T 4 (g(s)-\bar{g}_{\Delta_n}(s)) dB(s) - \int_0^T 8 (g(s)-\bar{g}_{\Delta_n}(s))^2 ds}] \leq 1.
\end{equation}
and by (\ref{diff-2}), we deduce that there exists $L_3 > 0$ such that
$$
\hspace{-1.5cm} \EX[e^{6 \int_0^T (g(s)-\bar{g}_{\Delta_n}(s))^2 ds}] \leq e^{L_3 \delta_{\Delta_n}^2} \EX[e^{L_3 \|B_0^{\delta_{\Delta_n}}\|^2}] \EX[e^{L_3 \delta^2_{\Delta_n} \|Y_0^T\|^2}].
$$
Note that it follows from Doob's submartingale inequality that if $\delta_{\Delta_n}$ is small enough,
$$
\EX[e^{L_3 \|B_0^{\delta_{\Delta_n}}\|^2}] < \infty.
$$
Furthermore, via a usual argument with the Gronwall inequality, we infer that there exists $L_2 > 0$ such that
$$
\|Y_0^T\| \leq L_2 (1+\|B_0^T\|),
$$
which, together with Doob's submartingale inequality, implies that if $\delta_{\Delta_n}$ is small enough,
$$
\EX[e^{L_3 \delta^2_{\Delta_n} \|Y_0^T\|^2}] < \infty.
$$
It then follows that for the first term in (\ref{two-terms})
\begin{equation} \label{SongJianMethod}
\EX[e^{2 (A_1(M, Y_0^T)-A_2(\Delta_n, M, Y_0^T))}] < \infty.
\end{equation}
A completely parallel argument will yield that for the second term in (\ref{two-terms})
$$
\EX\left[e^{-2(A_1(M, Y_0^T)-A_2(\Delta_n, M, Y_0^T))}\right] < \infty,
$$
which, together with (\ref{SongJianMethod}), immediately implies (\ref{conv-2a}), which in turn implies (\ref{conv-2}), as desired.

{\bf Step $\bf 1.2$.} In this step, we prove that as $n$ tends to infinity,
\begin{equation} \label{Gn-Conv}
\log \EX[e^{A_1(M, Y_0^T)}|Y(\Delta_n)] \to \log \EX[e^{A_1(M, Y_0^T)}|Y_0^T],
\end{equation}
in probability.

First, note that by Theorem $7.23$ of~\cite{li01}, we have,
$$
\frac{d \mu_{Y}}{d\mu_B}(Y_0^T)  = \int \frac{d \mu_{Y|M}}{d\mu_B}(Y_0^T|m) d\mu_M(m),
$$
where
$$
\frac{d \mu_{Y|M}}{d\mu_B}(Y_0^T|m) = e^{\int_0^T g(s, m, Y_0^s) dY(s)-\frac{1}{2} \int_0^T g^2(s, m, Y_0^s) ds}.
$$
It then follows from (\ref{RN-1}) that
\begin{align*}
\log \EX[e^{A_1(M, Y_0^T)}|Y_0^T] & = -\log \int \frac{d \mu_{Y|M}}{d\mu_B}(Y_0^T|m) d\mu_M(m)\\
& = -\log \int e^{-A_1(m, Y_0^T)} d \mu_M(m).
\end{align*}
Similarly, we have
\begin{align*}
\frac{d\mu_{Y(\Delta_n)}}{d\mu_{B(\Delta_n)}}(Y(\Delta_n)) &= \int \frac{d\mu_{Y(\Delta_n)|M}}{d\mu_{B(\Delta_n)}}(Y(\Delta_n)|m) d\mu_M(m)\\
                                      &= \int \frac{1}{\EX[e^{A_1(M, Y_0^T)}|Y(\Delta_n), m]} d\mu_M(m).
\end{align*}
It then follows from (\ref{RN-2}) that
$$
\log \EX[e^{A_1(M, Y_0^s)}|Y(\Delta_n)] = -\log \int \frac{1}{\EX[e^{A_1(M, Y_0^T)}|Y(\Delta_n), m]} d\mu_M(m).
$$

Now, we consider the following difference:
\begin{align*}
&\hspace{-1cm} \int e^{-A_1(m, Y_0^T)} d\mu_M(m)-\int \frac{1}{\EX\left[ \left. e^{A_1(m, Y_0^T)} \right| Y(\Delta_n), m\right]} d\mu_M(m)\\
&=\int e^{-A_1(m, Y_0^T)}-e^{-A_2(\Delta_n, m, Y_0^T)} d\mu_M(m)\\
&\quad + \int e^{-A_2(\Delta_n, m, Y_0^T)} \times \frac{\EX[e^{A_1(m, Y_0^T)-A_2(\Delta_n, m, Y_0^t)}|Y(\Delta_n), m]-1}{\EX[e^{A_1(m, Y_0^T)-A_2(\Delta_n, m, Y_0^t)}|Y(\Delta_n), m]} d\mu_M(m).
\end{align*}
Applying the inequality that for any $x, y \in \mathbb{R}$,
\begin{equation} \label{expo-inequality}
|e^x-e^y| \leq |x-y| (e^x \vee e^y) \leq |x-y| (e^x+e^y),
\end{equation}
we have, for some positive constant $L$,
\begin{align*}
&\hspace{-1cm}\EX\left[\left|\int e^{-A_1(m, Y_0^T)}-e^{-A_2(\Delta_n, m, Y_0^T)} d\mu_M(m)\right|\right] \\
&\leq \int \EX\left[\left|e^{-A_1(m, Y_0^T)}-e^{-A_2(\Delta_n, m, Y_0^T)} \right| \right] d\mu_M(m)\\
&\leq \int \EX\left[\left| A_1(m, Y_0^T)-A_2(m, Y_0^T)\right| \left(e^{-A_1(m, Y_0^T)}+e^{-A_2(m, Y_0^T)} \right) \right] d\mu_M(m)\\
&\leq \int \EX\left[\left(\left| \int_0^T (g(m)-\bar{g}_{\Delta_n}(m)) dB(s) \right|+\left|\int_0^T (g(m)-\bar{g}_{\Delta_n}(m))(g(s)-\frac{1}{2}g(m)-\frac{1}{2}\bar{g}_{\Delta_n}(m)) ds\right|\right) \right.\\
&\quad \times \left. \left(e^{-A_1(m, Y_0^T)}+e^{-A_2(m, Y_0^T)} \right) \right] d\mu_M(m)\\
&\leq \int \EX\left[\left(\left| \int_0^T (g(m)-\bar{g}_{\Delta_n}(m)) dB(s) \right| \right. \right.\\
&\quad \left. \left. + (L \delta_{\Delta_n}+ L \delta_{\Delta_n} +L \delta_{\Delta_n} \|Y_0^T\| +\sup_{|s-t| \leq \delta_{\Delta_n}} |B(s)-B(t)|) \int_0^T \left| g(s)-\frac{1}{2}g(m)-\frac{1}{2}\bar{g}_{\Delta_n}(m) \right| ds\right) \right.\\
&\quad \times \left. \left(e^{-A_1(m, Y_0^T)}+e^{-A_2(m, Y_0^T)} \right) \right] d\mu_M(m),
\end{align*}
where we have rewritten $g(s, m, Y_0^s)$ as $g(m)$, $\bar{g}_{\Delta_n}(s, m, \bar{Y}_{\Delta_n, 0}^s)$ as $\bar{g}_{\Delta_n}(m)$ for notational simplicity.

Now, using (\ref{diff-2}) and the It\^{o} isometry, we deduce that
\begin{equation} \label{prev-1}
\lim_{n \to \infty} \int \EX\left[\left| \int_0^T g(m)-\bar{g}_{\Delta_n}(m) dB(s)\right|^2 \right] d\mu_{M}(m) = 0,
\end{equation}
and
\begin{equation} \label{prev-2}
\lim_{n \to \infty} \int \EX[(L \delta_{\Delta_n}+ L \delta_{\Delta_n} + L \delta_{\Delta_n} \|Y_0^T\|+\sup_{|s-t| \leq \delta_{\Delta_n}} |B(s)-B(t)|)^2] d\mu_{M}(m) = 0.
\end{equation}
Now, using a similar argument as above with (\ref{case-1}) and Lemma~\ref{improved-liptser-1}, we can show that for any constant $K$,
\begin{equation} \label{prev-3}
\EX[e^{\int_0^T K \bar{g}_{\Delta_n}^2(s) ds}] = \EX[e^{\int_0^T K (\bar{g}_{\Delta_n}(s)-g(s)+g(s))^2 ds}] = \EX[e^{\int_0^T K (2(\bar{g}_{\Delta_n}(s)-g(s))^2+2g^2(s)) ds}]< \infty,
\end{equation}
provided that $n$ is large enough, which, coupled with a similar argument as in the derivation of (\ref{SongJianMethod}), proves that for $n$ large enough,
\begin{equation}  \label{prev-4}
\hspace{-1cm} \int \EX\left[ \left(e^{-A_1(m, Y_0^T)}+e^{-A_2(\Delta_n, m, Y_0^T)}\right)^2 \right] d\mu_M(m) < \infty,
\end{equation}
and furthermore
\begin{equation} \label{prev-5}
\int \EX\left[\left(\int_0^T \left| g(s)-\frac{1}{2}g(m)-\frac{1}{2}\bar{g}_{\Delta_n}(m) \right| ds\right)^2 \times \left(e^{-A_1(m, Y_0^T)}+e^{-A_2(\Delta_n, m, Y_0^T)} \right)^2\right] d\mu_M(m) < \infty.
\end{equation}
It then immediately follows that
\begin{equation}  \label{S-prime}
\lim_{n \to \infty} \EX\left[\left|\int e^{-A_1(m, Y_0^T)}-e^{-A_2(\Delta_n, m, Y_0^T)} d\mu_M(m)\right|\right] = 0.
\end{equation}

Now, using the shorthand notations as before, we have
{\small \begin{align*}
& \hspace{-1cm} \EX\left[\left| \int e^{-A_2(\Delta_n, m, Y_0^T)} \frac{\EX[e^{A_1(m, Y_0^T)-A_2(\Delta_n, m, Y_0^T)}|Y(\Delta_n), m]-1}{\EX[e^{A_1(m, Y_0^T)-A_2(\Delta_n, m, Y_0^T)}|Y(\Delta_n), m]} d\mu_M(m) \right| \right] \\
& =\EX\left[\left|\int \frac{\EX[e^{A_1(m, Y_0^T)-A_2(\Delta_n, m, Y_0^T)}-1|Y(\Delta_n), m]}{\EX[e^{A_1(m, Y_0^T)}|Y(\Delta_n), m]} d\mu_M(m) \right|\right] \\
& \leq \EX\left[\int \frac{\EX[|e^{A_1(m, Y_0^T)-A_2(\Delta_n, m, Y_0^T)}-1||Y(\Delta_n), m]}{\EX[e^{A_1(m, Y_0^T)}|Y(\Delta_n), m]} d\mu_M(m)\right] \\
& \leq \EX\left[\int \EX[|-A_1(m, Y_0^T)+A_2(\Delta_n, m, Y_0^T)| e^{|-A_1(m, Y_0^T)+A_2(\Delta_n, m, Y_0^T)|} | Y(\Delta_n), m] \right. \\
&\quad \times \left. \EX[e^{-A_1(m, Y_0^T)}|Y(\Delta_n), m] d\mu_M(m)\right]\\
& = \EX\left[\EX[|-A_1(m, Y_0^T)+A_2(\Delta_n, m, Y_0^T)| e^{|-A_1(m, Y_0^T)+A_2(\Delta_n, m, Y_0^T)|} | Y(\Delta_n), M] \right.\\
&\quad \times \left. \EX[e^{-A_1(m, Y_0^T)}|Y(\Delta_n), M] \left(\frac{d\mu_Y}{d\mu_B}(Y_0^T)\right)/\left(\frac{d\mu_{Y|M}}{d\mu_B}(Y_0^T|M)\right)\right]\\
& = \EX\left[|-A_1(m, Y_0^T)+A_2(\Delta_n, m, Y_0^T)| e^{|-A_1(m, Y_0^T)+A_2(\Delta_n, m, Y_0^T)|} \EX[e^{-A_1(m, Y_0^T)}|Y(\Delta_n), M] \right.\\
& \quad \times \left. \EX\left[\left(\frac{d\mu_Y}{d\mu_B}(Y_0^T)\right)/\left(\frac{d\mu_{Y|M}}{d\mu_B}(Y_0^T|M)\right) | Y(\Delta_n), M \right]\right].
\end{align*}}
Now, a similar argument as in (\ref{prev-1})-(\ref{prev-5}), together with the well-known fact (see, e.g., Theorem $6.2.2$ in~\cite{ih93}) that
$$
\frac{d\mu_Y}{d\mu_B}(Y_0^T)=e^{\int_0^T \EX[g(s)|Y_0^s] dY(s)-\frac{1}{2} \int_0^T \EX^2[g(s)|Y_0^s] ds}, \quad \frac{d\mu_{Y|M}}{d\mu_B}(Y_0^T|M)=e^{\int_0^T g(s) dY(s)-\frac{1}{2} \int_0^T g^2(s) ds},
$$
yields that
$$
\lim_{n \to \infty} \EX\left[\left| \int e^{-A_2(\Delta_n, m, Y_0^T)} \frac{\EX[e^{A_1(m, Y_0^T)-A_2(\Delta_n, m, Y_0^T)}|Y(\Delta_n), m]-1}{\EX[e^{A_1(m, Y_0^T)-A_2(\Delta_n, m, Y_0^T)}|Y(\Delta_n), m]} d\mu_M(m) \right| \right] = 0.
$$
Now, we are ready to conclude that as $n$ tends to infinity,
$$
\EX[e^{-\int_0^T g(s) dY(s)+\frac{1}{2} \int_0^T g^2(s) ds}|Y(\Delta_n)] \to \EX[e^{-\int_0^T g(s) dY+\frac{1}{2} \int_0^T g^2(s) ds}|Y_0^T]
$$
in probability and furthermore (\ref{Gn-Conv}), as desired.

{\bf Step $\bf 1.3$.} In this step, we will show the convergence of $\{\EX[\log \EX[e^{A_1(M, Y_0^T)}|Y(\Delta_n), M]]\}$ and $\{\EX[\log \EX[e^{A_1(M, Y_0^T)}|Y(\Delta_n)]]\}$ and further establish the theorem under the condition (\ref{case-1}).

First of all, using the concavity of the $\log$ function and the fact that $\log x \leq x$, we can obtain the following bounds:
$$
|\log \EX[e^{A_1(M, Y_0^T)}|Y(\Delta_n), M]| \leq \left|\EX\left[A_1(M, Y_0^T)| Y(\Delta_n), M \right]\right|
+\EX\left[\left. e^{A_1(M, Y_0^T)}\right|Y(\Delta_n), M \right],
$$
and
$$
|\log \EX[e^{A_1(M, Y_0^T)}|Y(\Delta_n)]| \leq \left|\EX\left[A_1(M, Y_0^T)| Y(\Delta_n) \right]\right|+\EX\left[\left. e^{A_1(M, Y_0^T)}\right|Y(\Delta_n) \right].
$$
And furthermore, using a similar argument as in {\bf Step} $\bf 1.1$, we can show that,
\begin{align*}
\EX\left[A_1(M, Y_0^T)| Y(\Delta_n), M \right] & = A_2(\Delta_n, M, Y_0^T) \times \EX\left[A_1(M, Y_0^T)-A_2(\Delta_n, M, Y_0^T)| Y(\Delta_n), M \right]\\
& \stackrel{n \to \infty}{\longrightarrow} A_1(M, Y_0^T)
\end{align*}
in probability, and
\begin{align*}
\EX[e^{A_1(M, Y_0^T)}|Y(\Delta_n), M] &= e^{A_2(\Delta_n, M, Y_0^T)} \EX[e^{A_1(M, Y_0^T)-A_2(\Delta_n, M, Y_0^T)}|Y(\Delta_n), M]\\
& \stackrel{n \to \infty}{\longrightarrow} e^{A_1(M, Y_0^T)}
\end{align*}
in probability. It then follows from the general Lebesgue dominated convergence theorem that
$$
\lim_{n \to \infty} \EX[\log \EX[e^{A_1(M, Y_0^T)}|Y(\Delta_n), M]] = \EX\left[A_1(M, Y_0^T)\right].
$$
A parallel argument can be used to show that
$$
\lim_{n \to \infty} \EX[\log \EX[e^{A_1(M, Y_0^T)}|Y(\Delta_n)]] = \EX[\log \EX[e^{A_1(M, Y_0^T)}|Y_0^T]].
$$
So, under the condition (\ref{case-1}), we have shown that
$$
\lim_{n \to \infty} I(M; Y(\Delta_n))=I(M; Y_0^T).
$$

{\bf Step $\bf 2$.} In this step, we will use the convergence in \textbf{Step} $\bf 1$ and establish the theorem without the condition (\ref{case-1}).

Following Page $264$ of~\cite{li01}, we define, for any $k$,
\begin{equation} \label{stopping-time}
\tau_k = \begin{cases}
\inf \{t \leq T: \int_0^t g^2(s, M, Y_0^s) ds \geq k\}, \mbox{ if } \int_0^T g^2(s, M, Y_0^s) ds \geq k,\\
T, \mbox{ if } \int_0^T g^2(s, M, Y_0^s) ds < k.
\end{cases}
\end{equation}
Then, we again follow~\cite{li01} and define a truncated version of $g$ as follows:
\begin{equation*}
g_{(k)}(t, \gamma_0^t, \phi_0^t)=g(t, \gamma_0^t, \phi_0^t) \mathbbm{1}_{\int_0^t g^2(s, \gamma_0^t, \phi_0^s) ds < k}.
\end{equation*}
Now, define a truncated version of $Y$ as follows:
\begin{equation*}
Y_{(k)}(t)=\int_0^t g_{(k)}(s, M, Y_0^s) ds +B(t), \quad t \in [0, T],
\end{equation*}
which, as elaborated on Page $265$ in~\cite{li01}, can be rewritten as
\begin{equation} \label{fixed-n}
Y_{(k)}(t)=\int_0^t g_{(k)}(s, M, Y_{(k), 0}^s) ds +B(t), \quad t \in [0, T].
\end{equation}
Note that for any fixed $k$, the system in (\ref{fixed-n}) satisfies the condition (\ref{case-1}), and so the theorem holds true. To be more precise, we have
$$
I(M; Y_{(k), 0}^{T})=\EX\left[\log \frac{d\mu_{Y_{(k)}|M}}{d\mu_{B}}(Y_{(k), 0}^{T}|M)\right]-\EX\left[\log \frac{d\mu_{Y_{(k)}}}{d\mu_{B}}(Y_{(k), 0}^T)\right],
$$
where
$$
\frac{d\mu_{Y_{(k)}|M}}{d\mu_{B}}(Y_{(k), 0}^T|M)=e^{\int_0^{T} g_{(k)}(s) dY_{(k)}(s)-\frac{1}{2} \int_0^{T} g_{(k)}^2(s)ds}=e^{\int_0^{\tau_k} g(s) dY(s)-\frac{1}{2} \int_0^{\tau_k} g^2(s)ds},
$$
and
$$
\frac{d\mu_{Y_{(k)}}}{d\mu_{B}}(Y_{(k), 0}^{T})=e^{\int_0^{T} \hat{g}_{(k)}(s) dY_{(k)}(s)-\frac{1}{2} \int_0^{T} \hat{g}_{(k)}^2(s) ds}=e^{\int_0^{\tau_k} \hat{g}(s) dY(s)-\frac{1}{2} \int_0^{\tau_k} \hat{g}^2(s) ds},
$$
where
$$
\hat{g}_{(k)}(s)=\EX[g_{(k)}(s, M, Y_0^s)|Y_{(k), 0}^s], \quad \hat{g}(s)=\EX[g(s, M, Y_0^s)|Y_0^s].
$$
It then follows from straightforward computations that
$$
I(M; Y_{(k), 0}^T) =\frac{1}{2} \EX\left[\int_0^{\tau_k} (g(s)-\hat{g}(s))^2 ds\right].
$$
Notice that it can be easily verified that $\tau_k \to T$ as $k$ tends to infinity, which, together with the monotone convergence theorem, further yields that monotone increasingly,
$$
I(M; Y_{(k), 0}^T)=\frac{1}{2} \EX\left[\int_0^{\tau_k} (g(s)-\hat{g}(s))^2 ds\right] \to I(M; Y_0^T)=\frac{1}{2} \EX\left[\int_0^T (g(s)-\hat{g}(s))^2 ds\right],
$$
as $k$ tends to infinity. By \textbf{Step} $\bf 1$, for any fixed $k_i$,
$$
\lim_{n \to \infty} I(M; Y_{(k_i)}(\Delta_n))=\lim_{n \to \infty} I(M; Y_{(k_i)}(\Delta_n \cap [0, \tau_{k_i}]))=I(M; Y_{(k_i), 0}^T),
$$
which means that there exists a sequence $\{n_i\}$ such that, as $i$ tends to infinity, we have, monotone increasingly,
$$
I(M; Y_{(k_i)}(\Delta_{n_i} \cap [0, \tau_{k_i}])) \to I(M; Y_{0}^{T}).
$$
Since, by the fact that $Y_{(k_i), 0}^T$ coincides with $Y_0^T$ on $t \in [0, \tau_{k_i} \wedge T]$ and then (\ref{I<=I}), we have
$$
I(M; Y(\Delta_{n_i})) \geq I(M; Y_{(k_i)}(\Delta_{n_i} \cap [0, \tau_{k_i}])).
$$
Now, using the fact that
$$
I(M; Y(\Delta_{n_i})) \leq I(M; Y_0^T),
$$
we conclude that as $i$ tends to infinity,
$$
\lim_{i \to \infty} I(M; Y(\Delta_{n_i})) = I(M; Y_{0}^T).
$$
A similar argument can be readily applied to any subsequence of $\{I(M; Y(\Delta_n))\}$, which will establish the existence of its further subsubsequence that converges to $I(M; Y_0^T)$, which implies that
$$
\lim_{n \to \infty} I(M; Y(\Delta_n))=I(M; Y_0^T),
$$
completing the proof of the theorem.

\section{Proof of Theorem~\ref{approximation-theorem}} \label{proof-approximation-theorem}

Throughout the proof, we will rewrite $t^{(n)}_i$ as $t_i$ for notational simplicity. As in the proof of Theorem~\ref{sampling-theorem}, we will again write the summation $\sum_m (\cdot) p_M(m)$ as the integral $\int (\cdot) d\mu_M(m)$. We proceed in the following two steps.

{\bf Step $\bf 1$.} In this step, we establish the theorem assuming that there exists a constant $C > 0$ such that for all $m \in \mathcal{M}$ and all $y_0^T \in C[0, T]$,
\begin{equation} \label{approximating-case-1}
\int_0^T g^2(s, m, y_0^s) ds < C.
\end{equation}

We first note that straightforward computations yield
\begin{equation} \label{YgivenM}
\hspace{-0.5cm} f_{Y^{(n)}(\Delta_n)|M}(Y^{(n)}(\Delta_n)|M)=\prod_{i=1}^{n} \frac{1}{\sqrt{2\pi(t_i-t_{i-1})}}\exp \left({-\frac{(Y^{(n)}(t_i)-Y^{(n)}(t_{i-1})- \int_{t_{i-1}}^{t_i} g(s, M, Y_0^{(n), t_{i-1}})ds)^2}{2(t_i-t_{i-1})}} \right),
\end{equation}
and
\begin{equation} \label{justY}
\hspace{-1cm} f_{Y^{(n)}(\Delta_n)}(Y^{(n)}(\Delta_n))=\int \prod_{i=1}^{n} \frac{1}{\sqrt{2\pi(t_i-t_{i-1})}}\exp \left({-\frac{(Y^{(n)}(t_i)-Y^{(n)}(t_{i-1})- \int_{t_{i-1}}^{t_i} g(s, m, Y^{(n), t_{i-1}}_{0}) ds)^2}{2(t_i-t_{i-1})}} \right) d\mu_M(m).
\end{equation}
With (\ref{YgivenM}) and (\ref{justY}), we have
\begin{align}
\hspace{-3mm} I(M; Y^{(n)}(\Delta_n)) & =\EX[\log f_{Y^{(n)}(\Delta_n)|M}(Y^{(n)}(\Delta_n)|M)]-\EX[\log f_{Y^{(n)}(\Delta_n)}(Y^{(n)}(\Delta_n))] \nonumber \\
&=\EX\left[-A_3(M, Y_0^{(n), T}) \right]-\EX\left[\log \int e^{-A_3(m, Y_0^{(n), T})} d\mu_M(m) \right], \label{I-M-Yn}
\end{align}
where
$$
A_3(m, Y_0^{(n), T}) \triangleq \sum_{i=1}^n \left({\frac{-2\int_{t_{i-1}}^{t_i} g(s, m, Y^{(n), t_{i-1}}_{0}) ds \; (Y^{(n)}(t_i)-Y^{(n)}(t_{i-1}))+ (\int_{t_{i-1}}^{t_i} g(s, m, Y^{(n), t_{i-1}}_{0}) ds)^2}{2(t_i-t_{i-1})}} \right),
$$
and
$$
A_3(M, Y_0^{(n), T}) \triangleq \sum_{i=1}^n \left({\frac{-2\int_{t_{i-1}}^{t_i} g(s, M, Y^{(n), t_{i-1}}_{0}) ds \; (Y^{(n)}(t_i)-Y^{(n)}(t_{i-1}))+ (\int_{t_{i-1}}^{t_i} g(s, M, Y^{(n), t_{i-1}}_{0}) ds)^2}{2(t_i-t_{i-1})}} \right).
$$
On the other hand, by (\ref{l-4}), we have
\begin{align}
I(M; Y_0^T) & =\EX\left[\log \frac{d\mu_{Y|M}}{d\mu_B}(Y_0^T|M)\right]-\EX\left[\log \frac{d\mu_{Y}}{d\mu_B}(Y_0^T)\right] \nonumber\\
& =\EX \left[\log e^{-A_1(M, Y_0^T)} \right]-\EX \left[\log \int e^{-A_1(m, Y_0^T) d\mu_M(m)} \right] \nonumber\\
& =\EX \left[-A_1(M, Y_0^T)\right] -\EX \left[\log \int e^{-A_1(m, Y_0^T) d\mu_M(m)} \right]. \label{I-M-Y}
\end{align}
Now, we compute
\begin{align*}
& \hspace{-1cm} \int_0^T g(s, M, Y_0^s) dY(s)-\sum_{i=1}^{n} \frac{\int_{t_{i-1}}^{t_i} g(s, M, Y_0^{(n), t_{i-1}}) ds (Y^{(n)}(t_i)-Y^{(n)}(t_{i-1}))}{t_i-t_{i-1}}\\
& =\int_0^T g(s, M, Y_0^s) dY(s)-\sum_{i=1}^{n} g(t_{i-1}, M, Y_0^{(n), t_{i-1}}) (Y^{(n)}(t_i)-Y^{(n)}(t_{i-1}))\\
& \quad -\sum_{i=1}^{n} \frac{\int_{t_{i-1}}^{t_i} (g(s, M, Y_0^{(n), t_{i-1}})-g(t_{i-1}, M, Y_0^{(n), t_{i-1}})) ds (Y^{(n)}(t_i)-Y^{(n)}(t_{i-1}))}{t_i-t_{i-1}}.
\end{align*}
It can be easily checked that the second term of the right hand side of the above equality converges to $0$ in mean. And for the first term, we have
\begin{align*}
&\hspace{-1cm} \int_0^T g(s, M, Y_0^s) dY(s)-\sum_{i=1}^{n} g(t_{i-1}, M, Y_0^{(n), t_{i-1}}) (Y^{(n)}(t_i)-Y^{(n)}(t_{i-1}))\\
&=\sum_{i=1}^n \int_{t_{i-1}}^{t_i} g(s, M, Y_0^s) dY(s)-\sum_{i=1}^n g(t_{i-1}, M, Y_0^{(n), t_{i-1}}) (Y(t_i)-Y(t_{i-1}))\\
&\quad +\sum_{i=1}^n g(t_{i-1}, M, Y_0^{(n), t_{i-1}}) ((Y(t_i)-Y(t_{i-1}))-(Y^{(n)}(t_i)-Y^{(n)}(t_{i-1})))\\
&=\sum_{i=1}^n \int_{t_{i-1}}^{t_i} g(s, M, Y_0^s) dY(s)-\sum_{i=1}^n \int_{i=1}^n g(t_{i-1}, M, Y_0^{(n), t_{i-1}}) dY(s)\\
&\quad +\sum_{i=1}^n g(t_{i-1}, M, Y_0^{(n), t_{i-1}}) ((Y(t_i)-Y(t_{i-1}))-(Y^{(n)}(t_i)-Y^{(n)}(t_{i-1})))\\
&=\sum_{i=1}^n \int_{t_{i-1}}^{t_i} (g(s, M, Y_0^s)-g(t_{i-1}, M, Y_0^{(n), t_{i-1}})) dY(s)\\
&\quad +\sum_{i=1}^n g(t_{i-1}, M, Y_0^{(n), t_{i-1}}) ((Y(t_i)-Y(t_{i-1}))-(Y^{(n)}(t_i)-Y^{(n)}(t_{i-1}))).
\end{align*}
It then follows from Conditions (d) and (e), Lemmas~\ref{improved-liptser-1},~\ref{improved-liptser-2} and~\ref{Y-Y} that
\begin{equation}  \label{first-half}
\EX\left[\left|\sum_{i=1}^{n} \frac{\int_{t_{i-1}}^{t_i} g(s, M, Y_0^{(n), t_{i-1}}) ds (Y^{(n)}(t_i)-Y^{(n)}(t_{i-1}))}{t_i-t_{i-1}} - \int_0^T g(s, M, Y_0^s) dY(s)\right| \right]=O(\delta^{\frac{1}{2}}_{\Delta_n}).
\end{equation}
And using a similar argument as above, we deduce that
\begin{equation}  \label{second-half}
\EX\left[\left|\frac{1}{2} \sum_{i=1}^n \frac{(\int_{t_{i-1}}^{t_i} g(s, M, Y^{(n),t_{i-1}}_{0})ds)^2}{t_i-t_{i-1}}-\frac{1}{2} \int_0^T g(s, M, Y_0^s)^2 ds \right|\right] = O(\delta^{\frac{1}{2}}_{\Delta_n}).
\end{equation}
It then follows from (\ref{first-half}) and (\ref{second-half}) that
\begin{equation} \label{M-conv}
\EX\left[ \left| A_3(M, Y_0^{(n), T}) - A_1(M, Y_0^T) \right|\right] = O(\delta^{\frac{1}{2}}_{\Delta_n}).
\end{equation}

We now establish the following convergence:
{\small \begin{equation} \label{m-conv}
\lim_{n \to \infty} \EX\left[\log \int e^{-A_3(m, Y_0^{(n), T})} d\mu_M(m) \right] = \EX\left[\log \int e^{-A_1(m, Y_0^T)} d\mu_M(m) \right].
\end{equation}}
Note that using a parallel argument as in the derivation of (\ref{M-conv}), we can establish
\begin{equation}  \label{m-conv-1}
\lim_{n \to \infty} \EX \int \left| -A_3(m, Y_0^{(n),T})-A_1(m, Y_0^T) \right| d\mu_M(m) = 0;
\end{equation}
and similarly as in the derivation of (\ref{S-prime}), from Conditions (d) and (e), Lemmas~\ref{improved-liptser-1},~\ref{improved-liptser-2} and~\ref{Y-Y}, we deduce that
\begin{equation} \label{m-conv-2}
\lim_{n \to \infty} \EX\left[\int \left| e^{-A_3(m, Y_0^{(n),T})}-e^{-A_1(m, Y_0^T)} \right| d\mu_M(m)\right] = 0.
\end{equation}
Moreover, we note that it always holds that
\begin{equation} \label{m-conv-3}
\left| \log \int e^{-A_3(m, Y_0^{(n),T})} d\mu_M(m) \right| \leq \int e^{-A_3(m, Y_0^{(n),T})} d\mu_M(m)+ \left| \int A_3(m, Y_0^{(n), T}) d\mu_M(m) \right|.
\end{equation}
Then, the desired (\ref{m-conv}) follows from an application of the general Lebesgue dominated convergence theorem with (\ref{m-conv-1}), (\ref{m-conv-2}) and (\ref{m-conv-3}).

Finally, with (\ref{I-M-Yn}), (\ref{I-M-Y}), (\ref{M-conv}) and (\ref{m-conv}), we conclude that
$$
\lim_{n \to \infty} I(M; Y^{(n)}(\Delta_n)) = I(M; Y_0^T),
$$
establishing the theorem with the extra condition (\ref{approximating-case-1}).

{\bf Step $\bf 2$.}  In this step, we will use the convergence in \textbf{Step} $\bf 1$ and establish the theorem without the condition (\ref{approximating-case-1}).

Defining the stopping $\tau_k$, $g_{(k)}$ and $Y_{(k)}$ as in the proof of 2) of Theorem~\ref{sampling-theorem}, we again have
\begin{equation*}
Y_{(k)}(t)=\int_0^t g_{(k)}(s, M, Y_{(k), 0}^s) ds +B(t), \quad t \in [0, T].
\end{equation*}
For any fixed $k$, applying the Euler-Maruyama approximation as in (\ref{Euler-Maruyama-Sequence}) and (\ref{linear-interpolation}) to the above channel with respect to $\Delta_n$, we obtain the process $\{Y_{(k)}^{(n)}(t)\}$.

Now, recall that $I(M; Y^{(n)}(\Delta_n))$ can be computed as in (\ref{I-M-Yn}), and moreover, by some straightforward computations,
\begin{align*}
\EX[-A_3(M, Y_0^{(n), T})] = \EX\left[\sum_{i=1}^n {\frac{(\int_{t_{i-1}}^{t_i} g(s, M, Y^{(n), t_{i-1}}_{0}) ds)^2}{2(t_i-t_{i-1})}}\right].
\end{align*}
Similarly we have
\begin{align*}
I(M; Y_{(k)}^{(n)}(\Delta_n))& =\EX[\log f_{Y_{(k)}^{(n)}(\Delta_n)|M}(Y_{(k)}^{(n)}(\Delta_n)|M)]-\EX[\log f_{Y_{(k)}^{(n)}(\Delta_n)}(Y_{(k)}^{(n)}(\Delta_n))] \\
&= \EX[-A_4(M, Y_{(k), 0}^{(n), T})] - \EX\left[\log \int e^{-A_4(m, Y_{(k), 0}^{(n), T})} d\mu_M(m)\right],
\end{align*}
where
$$
A_4(m, Y_{(k), 0}^{(n), T})=\sum_{i=1}^n \left({\frac{-2\int_{t_{i-1}}^{t_i} g_{(k)}(s, m, Y^{(n), t_{i-1}}_{(k),0}) ds(Y_{(k)}^{(n)}(t_i)-Y_{(k)}^{(n)}(t_{i-1}))+ (\int_{t_{i-1}}^{t_i} g_{(k)}(s, m, Y^{(n), t_{i-1}}_{(k), 0}) ds)^2}{2(t_i-t_{i-1})}} \right),
$$
$$
A_4(M, Y_{(k), 0}^{(n), T})=\sum_{i=1}^n \left({\frac{-2\int_{t_{i-1}}^{t_i} g_{(k)}(s, M, Y^{(n), t_{i-1}}_{(k),0}) ds(Y_{(k)}^{(n)}(t_i)-Y_{(k)}^{(n)}(t_{i-1}))+ (\int_{t_{i-1}}^{t_i} g_{(k)}(s, M, Y^{(n), t_{i-1}}_{(k), 0}) ds)^2}{2(t_i-t_{i-1})}} \right),
$$
and moreover, by some straightforward computations,
$$
\EX[-A_4(M, Y_{(k),0}^{(n),T})] = \EX\left[\sum_{i=1}^n {\frac{(\int_{t_{i-1}}^{t_i} g_{(k)}(s, M, Y_{(k), 0}^{(n), t_{i-1}}) ds)^2}{2(t_i-t_{i-1})}}\right].
$$
Note that it can be easily verified that
$$
\frac{1}{f(Y^{(n)}(\Delta_n)|Y_{(k)}^{(n)}(\Delta_n), \tau_k)}=\EX\left[\left.\frac{1}{f(Y^{(n)}(\Delta_n)|Y_{(k)}^{(n)}(\Delta_n), \tau_k, M)}\right|Y^{(n)}(\Delta_n), \tau_k \right],
$$
which boils down to
$$
\frac{\int e^{-A_4(m, Y_{(k),0}^{(n),T})} d\mu_M(m)}{\int e^{-A_3(m, Y_0^{(n), T})} d\mu_M(m)} =\EX\left[\left. e^{A_3(M, Y_0^{(n), T})-A_4(M, Y_{(k),0}^{(n),T})} \right|Y^{(n)}(\Delta_n), \tau_k\right].
$$
Using this and Jensen's inequality, we deduce that
\begin{align}
\EX\left[\log \frac{\int e^{-A_4(m, Y_{(k),0}^{(n),T})} d\mu_M(m)}{\int e^{-A_3(m, Y_0^{(n), T})} d\mu_M(m)}\right] & = \EX \left[\log \EX\left[\left. e^{A_3(M, Y_0^{(n), T})-A_4(M, Y_{(k),0}^{(n),T})}\right|Y^{(n)}(\Delta_n), \tau_k \right] \right] \nonumber \\
& \leq \log \EX\left[e^{A_3(M, Y_0^{(n), T})-A_4(M, Y_{(k),0}^{(n),T})}\right] \nonumber\\
& \leq 0, \label{tricky}
\end{align}
where for the last inequality, we have applied Fatou's lemma as in deriving (\ref{Fatou-Lemma}).

Now, using (\ref{I<=I}) and the fact that $Y^{(n)}$ and $Y^{(n)}_{(k)}$ coincide over $[0, \tau_k \wedge T]$, we infer that
\begin{equation} \label{bigger-than-stopped}
I(M; Y^{(n)}(\Delta_n))-I(M; Y^{(n)}_{(k)}(\Delta_n)) \geq 0,
\end{equation}
and furthermore, one verifies that for any $\varepsilon > 0$,
\begin{align*}
& \hspace{-1cm} I(M; Y^{(n)}(\Delta_n))-I(M; Y^{(n)}_{(k)}(\Delta_n)) \\
& = (\EX[-A_3(M, Y_0^{(n), T})] - \EX[-A_4(M, Y_{(k),0}^{(n),T})]) \\
& \quad - \left(\EX\left[\log \int e^{-A_3(m, Y_0^{(n), T})} d\mu_M(m)\right] - \EX\left[\log \int e^{-A_4(M, Y_{(k),0}^{(n),T})} d\mu_M(m)\right]\right)\\
& \stackrel{(a)}{\leq} \EX\left[\sum_{i=1}^n {\frac{(\int_{t_{i-1}}^{t_i} g(s, M, Y^{(n), t_{i-1}}_{0}) ds)^2}{2(t_i-t_{i-1})}}\right] - \EX\left[\sum_{i=1}^n {\frac{(\int_{t_{i-1}}^{t_i} g_{(k)}(s, M, Y^{(n), t_{i-1}}_{(k), 0}) ds)^2}{2(t_i-t_{i-1})}}\right]\\
& = \EX\left[\sum_{i=1}^n {\frac{(\int_{t_{i-1}}^{t_i} g(s, M, Y^{(n), t_{i-1}}_{0}) ds)^2}{2(t_i-t_{i-1})}}\right] - \EX\left[\sum_{i=1}^n {\frac{(\int_{t_{i-1} \wedge \tau_k}^{t_i \wedge \tau_k} g(s, M, Y^{(n), t_{i-1}}_{0}) ds)^2}{2(t_i-t_{i-1})}}\right]\\
& = \EX\left[\sum_{i=1}^n {\frac{(\int_{t_{i-1} \vee \tau_k}^{t_i \vee \tau_k} g(s, M, Y^{(n), t_{i-1}}_{0}) ds)^2}{2(t_i-t_{i-1})}}\right]\\
& \leq \EX\left[\int_{\tau_k}^{T} g(s, M, Y^{(n), \lfloor s \rfloor_{\Delta_n}}_{0})^2 ds\right]\\
& \leq \EX\left[\int_{\tau_k}^{T} g(s, M, Y^{(n), \lfloor s \rfloor_{\Delta_n}}_{0})^2 ds; T-\tau_k \leq \varepsilon \right]+\EX\left[\int_{\tau_k}^{T} g(s, M, Y^{(n), \lfloor s \rfloor_{\Delta_n}}_{0})^2 ds; T-\tau_k > \varepsilon \right]\\
& \leq \int_{T-\varepsilon}^{T} \EX\left[ g(s, M, Y^{(n), \lfloor s \rfloor_{\Delta_n}}_{0})^2 \right] ds +\EX\left[\int_{\tau_k}^{T} g(s, M, Y^{(n), \lfloor s \rfloor_{\Delta_n}}_{0})^2 ds; T-\tau_k > \varepsilon \right].
\end{align*}
where we have used (\ref{tricky}) for (a) and $\lfloor s \rfloor_{\Delta_n}$ denotes the unique number $n_0$ such that $t_{n_0} \leq s < t_{n_0+1}$. Using the easily verifiable fact that $\tau_k \to T$ almost surely as $k$ tends to infinity, (\ref{bigger-than-stopped}) and the fact that $\varepsilon$ can be arbitrarily small, we conclude that as $k$ tends to infinity, uniformly over all $n$,
\begin{equation} \label{absolute-value-bounded}
I(M; Y^{(n)}_{(k)}(\Delta_n)) \to I(M; Y^{(n)}(\Delta_n)).
\end{equation}
Next, an application of the monotone convergence theorem with the fact that $\tau_k \to T$ as $k$ tends to infinity yields that monotone increasingly
$$
I(M; Y_{(k), 0}^T)=\frac{1}{2} \EX\left[\int_0^{\tau_k} (g(s)-\hat{g}(s))^2 ds\right] \to I(M; Y_0^T)=\frac{1}{2} \EX\left[\int_0^T (g(s)-\hat{g}(s))^2 ds\right]
$$
as $n$ tends to infinity. By \textbf{Step} $\bf 1$, for any fixed $k_i$,
$$
\lim_{n \to \infty} I(M; Y_{(k_i)}^{(n)}(\Delta_n))=I(M; Y_{(k_i), 0}^T),
$$
which means that there exists a sequence $\{n_i\}$ such that, as $i$ tends to infinity,
$$
I(M; Y_{(k_i)}^{(n_i)}(\Delta_{n_i})) \to I(M; Y_{0}^{T}).
$$
Moreover, by (\ref{absolute-value-bounded}),
$$
\lim_{i \to \infty } I(M; Y_{(k_i)}^{(n_i)}(\Delta_{n_i})) = \lim_{i \to \infty } I(M; Y^{(n_i)}(\Delta_n)),
$$
which further implies that
$$
\lim_{i \to \infty} I(M; Y^{(n_i)}(\Delta_{n_i}))=I(M; Y_0^T).
$$
The theorem then follows from a usual subsequence argument as in the proof of 2) of Theorem~\ref{sampling-theorem}.

\end{document}